\documentclass[a4paper,UKenglish,cleveref, autoref, thm-restate]{lipics-v2021}

\hideLIPIcs  

\usepackage{tikz}
\usepackage{stmaryrd}
\usepackage{xcolor}
\usetikzlibrary{calc}

\usepackage[normalem]{ulem}

\title{Online and Feasible Presentability: From Trees to Modal Algebras\footnote{This is a related version of the paper accepted at ICALP 2025.}} 

\author{Nikolay Bazhenov}{Novosibirsk State University, Novosibirsk, Russia}{nickbazh@yandex.ru}{https://orcid.org/0000-0002-5834-2770}{The work of N.~Bazhenov is supported by the Mathematical Center in Akademgorodok under the agreement 
No.~075-15-2022-282 with the Ministry of Science and Higher Education of the Russian Federation.}

\author{Dariusz Kalociński}{Institute of Computer Science, Polish Academy of Sciences, Warsaw, Poland}{dariusz.kalocinski@gmail.com}{https://orcid.org/0000-0002-3044-525X}{The work of D. Kalociński is supported by the National Science Centre Poland under the agreement no. 2023/49/B/HS1/03930.}

\author{Michał Wrocławski}{Faculty of Philosophy, University of Warsaw, Poland}{m.wroclawski@uw.edu.pl}{https://orcid.org/0000-0003-2679-7321}{}

\authorrunning{N. Bazhenov, D. Kalociński, and M. Wrocławski} 

\Copyright{Nikolay Bazhenov, Dariusz Kalociński, and Michał Wrocławski} 



\category{} 

\relatedversion{} 

\acknowledgements{We would like to thank Luca San Mauro and the reviewers for helpful comments.}

\nolinenumbers 

\EventEditors{John Q. Open and Joan R. Access}
\EventNoEds{2}
\EventLongTitle{42nd Conference on Very Important Topics (CVIT 2016)}
\EventShortTitle{CVIT 2016}
\EventAcronym{CVIT}
\EventYear{2016}
\EventDate{December 24--27, 2016}
\EventLocation{Little Whinging, United Kingdom}
\EventLogo{}
\SeriesVolume{42}
\ArticleNo{23}

\begin{document}

\maketitle

\begin{abstract}

We investigate whether every computable member of a given class of structures admits a fully primitive recursive (also known as punctual) or fully P-TIME copy. A class with this property is referred to as punctually robust or P-TIME robust, respectively. We present both positive and negative results for structures corresponding to well-known representations of trees, such as binary trees, ordered trees, sequential (or prefix) trees, and partially ordered (poset) trees. A corollary of one of our results on trees is that semilattices and lattices are not punctually robust. In the main result of the paper, we demonstrate that, unlike Boolean algebras, modal algebras—--that is, Boolean algebras with modality—--are not punctually robust. The question of whether distributive lattices are punctually robust remains open. The paper contributes to a decades-old program on effective and feasible algebra, which has recently gained momentum due to rapid developments in punctual structure theory and its connections to online presentations of structures.  
\end{abstract}

\section{Introduction}

Suppose an infinite graph \( G = (\mathbb{N}, E) \) is given via a computer program \( P \), which, for any pair \( x, y \in \mathbb{N} \), determines whether \( E(x, y) \) holds. This graph is an infinite structure represented in a finitary manner. Such structures are called computable.\footnote{An infinite countable structure $\mathcal A$ is computable if the domain of $\mathcal A$ is a computable subset of $\mathbb N$ and all the relations and functions from the signature of $\mathcal A$ are uniformly computable \cite{ash_computable_2000}.} However, without further restrictions on \( P \), the presentation of \( G \) via $P$ may be inefficient. Does there always exist a more efficient representation of $G$---that is a computer program \( P' \) such that \( P' \) computes \( E' \), \( G' = (\mathbb{N}, E') \) is isomorphic to \( G \), and \( E' \) is primitive recursive or even P-TIME computable?

The answer to the question posed above is negative \cite{kalimullin_algebraic_2017}. However, if we impose certain structural restrictions on $G$—--such as requiring $G$ to be strongly locally finite or to be a tree—--the answer becomes positive \cite{cenzer_feasible_1998}. A more general condition that guarantees the existence of a nice copy of $G$ has been presented in \cite{kalocinski_punctual_2024}.

The question raised in the opening paragraph admits several natural generalizations. For instance, one could replace graphs with other structures or explore alternative effectiveness conditions beyond primitive recursiveness or polynomial-time computability. Various instances of this question have drawn interest from researchers at the intersection of automata theory, complexity theory, and computable algebra. Notable examples include studies on algebraic structures represented by finite-state automata \cite{khoussainov1994automatic,Blumensath-Gradel-00,Gradel-20} and investigations into polynomial-time algebra \cite{Nerode-Remmel-89,cenzer_polynomial_1991,alaev_polynomial_2018}. Recently, Kalimullin, Melnikov, and Ng \cite{kalimullin_algebraic_2017} initiated a research program at the intersection of computability and complexity, where primitive recursiveness serves as a fundamental tool for representing structures. Henceforth, we adopt this paradigm in the paper. 

\begin{definition}[punctual structure \cite{kalimullin_algebraic_2017}]\label{def:punctual structure}
An infinite structure over a finite signature is \emph{punctual} if its domain is $\mathbb N$ and all of its relations and functions are primitive recursive.
\end{definition}
In the literature, punctual structures are also referred to as fully primitive recursive (fpr).

The purpose of the above definition was to capture the notion of an infinite structure which can be presented in an online fashion, or ``without delay''. The requirement that the domain be $\mathbb{N}$ (or any \emph{fixed} primitive recursive subset of $\mathbb N$) is significant. If, instead, we require the domain can be \emph{any} primitive recursive subset of $\mathbb N$, a distinct notion arises (cf. Theorem 1.2 in \cite{cenzer_polynomial_1991} and the associated discussion). More importantly, structures satisfying the relaxed definition may lack the typical online characteristics \cite{kalimullin_algebraic_2017,bazhenov_foundations_2019}.

Henceforth, all structures, unless stated otherwise, are countably infinite.

A more detailed explanation of the underlying nature of punctual structures follows. This will be useful later, as many subsequent constructions adhere to this general pattern.

Essentially, a punctual structure $\mathcal A$ may be given by a primitive recursive algorithm $P$ which, on input $\overline{x} = x_1,x_2,\dots,x_n$, determines $R^\mathcal{A}(\overline x)$ and $f^\mathcal{A}(\overline x)$, for every $R,f$ from the signature of $\mathcal A$. Intuitively, for arguments $x_i \leq s$ (where $s\in\mathbb{N}$), all the relations and function values are determined by the primitive recursive computation $P(\overline{x})$ which should converge by the  computational stage $s$. In a manner of speaking, our decision whether $R(\overline x)$, or what is $f(\overline x)$, must be \emph{quick}, or made \emph{without delay}. Here \emph{quick} and \emph{without delay} means \emph{now}, that is at most at the stage $s$. In fact, Definition \ref{def:punctual structure} allows for a delay, but the delay itself must be primitive recursive---any instance of a truly unbounded search is prohibited.

As observed above, a punctual presentation of a structure possesses the following core property: \emph{input is received and processed incrementally, and decisions about incoming data must be made without access to the complete set of problem data}. This property is a defining characteristic of online procedures in general, both in finite and infinite settings: in the finite setting, online algorithmics comprises a distinctive subfield of computer science, where performances of online algorithms are compared, relative to the offline performance, using competitive analysis (see, e.g., \cite{borodin_online_2005}); in the infinite setting, online graph coloring procedures from infinite combinatorics serve as a pertinent example (see, e.g., \cite{kierstead_recursive_1998}). Due to the core property punctual structures share with online algorithms, it has been suggested that they encapsulate, at least partially, what may intuitively be described as a truly online presentation of a countably infinite structure; the justification of this approach is a delicate task and has been accomplished elsewhere \cite{kalimullin_algebraic_2017,bazhenov_foundations_2019}. For alternative theoretical approaches, see \cite{downey_foundations_2021,askes_online_2022}.

The theory of online structures has rapidly emerged as an intriguing subfield of computable structure theory, yielding an impressive amount of results across various topics, including universality \cite{downey2020punctual,DowneyHKMT-20}, degree structures \cite{bazhenov2020online,greenberg2021non,MelNg-IJM}, and definability \cite{kalimullin_definability}, among others~\cite{PRA-24,Mel-Dorzhieva-23,MelNg-PAMS}. 
One of the central topics in online structure theory concerns punctual presentability.
\begin{definition}
A structure is \emph{punctually presentable} if it is isomorphic to a punctual one.
\end{definition}
This definition formalizes the property we asked about $G$ in the opening paragraph of this paper. More generally, given a class $\mathfrak K$ of structures, we ask whether every computable member of $\mathfrak K$ has a punctual copy.  \begin{definition}[punctual robustness \cite{kalocinski_punctual_2024}]
    We say that a class of structures $\mathfrak{K}$ is \emph{punctually robust}, if every computable member of $\mathfrak{K}$ is punctually presentable.
\end{definition}

Turning to more feasible computations, similar notions naturally arise in the polynomial-time setting. 
As usual, our underlying model for P-TIME computations is provided by Turing machines.

\begin{definition}[P-TIME structure and P-TIME robustness]
     By $\mathbb{B}$ we denote $\{0,1\}^{<\omega}$, i.e., the set of all finite binary strings. A structure $\mathcal{S}$ in a finite signature $L$ is \emph{P-TIME} (or polynomial-time computable) if the domain of $\mathcal{S}$ is a polynomial-time computable subset of $\mathbb{B}$, and all the $L$-relations and $L$-functions of $\mathcal{S}$ are polynomial-time computable. 
     A P-TIME structure $\mathcal{S}$ is \emph{fully P-TIME} if the domain of $\mathcal{S}$ is equal to $\mathbb{B}$. A structure is fully P-TIME presentable if it is isomorphic to a fully P-TIME structure.

     We say that a class of structures $\mathfrak{K}$ is \emph{P-TIME robust} if every computable member of $\mathfrak{K}$ is fully P-TIME presentable. 
\end{definition}

We note that for this framework, in place of the binary alphabet $\{0,1\}$ in the definition of P-TIME robustness, one can choose an arbitrary alphabet $\Sigma$ which contains at least two symbols (see, e.g., Chapter~3 in~\cite{Cenzer-Remmel-Survey-PTIME}).

A proof of punctual/P-TIME non-robustness for a given isomorphism-closed class $\mathfrak K$ of $L$-structures (where $L$ indicates the signature) usually requires specific strategies tailored to the type of structure we are dealing with. However, the general scheme is an example of diagonalization and can be framed in game-theoretic terms as follows. We imagine playing a game against infinitely many adversaries $\mathcal A_1, \mathcal A_2, \dots$. Adversaries are given via a computable enumeration. Each $\mathcal A_i$ is a punctual/P-TIME $L$-structure and $\mathcal A_1, \mathcal A_2,\dots$ contains all isomorphism-types of punctual/P-TIME copies of the structures from $\mathfrak K$. The proof must give a \emph{computable} strategy for building $\mathcal B \in \mathfrak K$ such that $\forall i (\mathcal A_i \in \mathfrak K \Rightarrow \mathcal B \not\cong \mathcal A_i)$.

 The following classes of structures have been shown to be punctually robust: equivalence structures \cite{cenzer_polynomial_1991,kalimullin_algebraic_2017}, (relational) successor trees \cite{cenzer_feasible_1998}, linear orders \cite{grigorieff_every_1990,kalimullin_algebraic_2017}, torsion-free abelian groups \cite{kalimullin_algebraic_2017}, Boolean algebras \cite{kalimullin_algebraic_2017}, abelian $p$-groups \cite{kalimullin_algebraic_2017}, graphs with infinite semitransversals \cite{kalocinski_punctual_2024}; on the other hand, there are computable undirected graphs \cite{kalimullin_algebraic_2017}, computable torsion abelian groups \cite{cenzer_polynomial-time_1992}, computable Archimedean ordered abelian groups \cite{kalimullin_algebraic_2017} and functional predecessor trees \cite{kalocinski_punctual_2024} with no punctual copy.

\paragraph*{Contributions}In this work we provide new results on punctual and P-TIME robustness of trees, (semi)lattices and modal algebras. Note that if a class is not punctually robust then it is not P-TIME robust either. Hence, P-TIME robustness is attested only for punctually robust classes. In the following, we briefly describe each of the three contributions.

 The first group of results concerns trees. It is a continuation of the work by Cenzer and Remmel on (variations of) relational successor trees \cite{cenzer_feasible_1998}, as well as of Kalociński, San Mauro and Wrocławski on functional predecessor trees \cite{kalocinski_punctual_2024}. Our results extend this line of research to a few other well-known tree structures.

One of the fundamental types of tree-like structures is an ordered tree (\cite{knuth_art_1968}, p. 306). This concept arises naturally by imposing an ordering on the children of each node in the tree. To avoid confusion with poset trees (to be defined), we use the term \emph{relational predecessor trees with ordering}. The definition could have used the successor relation instead, without affecting Theorem \ref{theo:predecessor-trees-with-ordering} (see below).

 \begin{definition}[r.p.o. tree]\label{def rpo tree}
     $\mathcal{T}=(T, P, <,r)$ is a relational predecessor tree with ordering (abbreviated r.p.o.~tree) if $P$ is a binary relation and $<$ is a partial ordering satisfying:
     \begin{enumerate}
         \item $(T,P)$ is a directed graph such that all edges are oriented towards the root $r$, and the underlying undirected graph is connected and acyclic,
    \item  for every $x,y \in T$, $x \leq y \vee y \leq x$ iff $x$ and $y$ have the same parent.
     \end{enumerate}
 \end{definition}

The relation $P$ is the immediate predecessor relation of the tree. In Section \ref{sec:rpo} we prove:

\begin{restatable}{theorem}{theoRPOtrees}\label{theo:predecessor-trees-with-ordering}
    The class of r.p.o.~trees is punctually robust and P-TIME robust.
\end{restatable}Theorem~\ref{theo:predecessor-trees-with-ordering} unifies and expands on both the results of Cenzer and Remmel \cite{cenzer_feasible_1998} on trees, and of Grigorieff \cite{grigorieff_every_1990} on linear orderings.

Another type of the tree that we consider is a poset tree. Given a partial order $(P, \leq)$, we define $P_{\leq x} = \{ y \in P: y \leq x\}$ and $P_{\geq x} = \{y \in P: y \geq x\}$. 
\begin{definition}[poset tree]\label{def:partially_ordered_tree}
    A partial order $(P,\leq)$ is a poset tree if $(P,\leq)$ has the greatest element (the root) and, for every $x \in P$, $P_{\geq x}$ is a finite linear order. 
\end{definition}
Such trees are well-known in set theory, where a tree is a partially ordered set $(P,<)$, usually with the least element, such that for each $x \in P$, the set $P_{<x}$ is well-ordered (see, e.g., Definition 9.10 in \cite{jech_set_2007}). The inessential difference between the two definitions is the direction in which the tree grows. A more important difference is that, in set theory, a node may be infinitely far apart from the root. Our definition does not allow it. However, the following theorem of ours, which we prove in Section \ref{sec:partially_ordered_trees}, works even for the class of trees in the set-theoretic sense (provided we reverse the growth direction):

\begin{restatable}{theorem}{thmPosetTrees}\label{thmPosetTrees}
    Poset trees are not punctually robust.
\end{restatable} While the underlying combinatorial idea behind the proof is quite intuitive, and some variants of it have been used elsewhere for a stronger notion of a tree \cite{cenzer_feasible_1998}, the present case requires much more care, as nodes in a poset tree do not carry information about their distances from the root. Even more importantly, Theorem \ref{thmPosetTrees} yields non-robustness of structures that lie intermediate between trees and Boolean algebras:

\begin{restatable}{corollary}{corLattices}\label{cor:lattices} The following classes of structures are not punctually robust:
\begin{romanenumerate}
    \item join semilattices and meet semilattices,\label{item semilattices} 
    \item lattices, complemented lattices, and non-distributive lattices.\label{item lattices}
\end{romanenumerate}
The result holds under both order-theoretic and algebraic interpretation. 
\end{restatable}

The remaining open question is whether the class of distributive lattices is punctually robust. The strategy used to obtain Theorem \ref{thmPosetTrees} is closely related to the `non-distributivity' of the computable poset tree we construct, and is therefore unlikely to provide a solution. To resolve the question, a different approach would be required, if distributive lattices are punctually robust at all.

The third contribution deals with modal algebras which are obtained from Boolean algebras by adding a modality operator. Boolean algebras constitute a classical object in mathematical logic. It is well-known that the class of Boolean algebras provides algebraic semantics for the classical propositional calculus (see, e.g., \cite{Rasiowa-Sikorski}). On the other hand, countable Boolean algebras are well-studied in computable structure theory---we refer to the monographs~\cite{Goncharov-Book,Downey-Melnikov-Book} for a detailed exposition of results in this area. 

Intuitively speaking, here we treat a \emph{modality} as a formal logical operator which expresses the possibility of formulas: $\diamondsuit p$ means that the statement $p$ is possible. More formally, one typically considers additional logical connectives $\diamondsuit$ (the possibility operator) and $\square$ (the necessity operator), and defines an appropriate modal propositional calculus which extends the classical propositional calculus (see, e.g., the monograph~\cite{CZ-Modal-Logic}). Modal algebras provide algebraic semantics for (normal) modal logics. The formal definition of a modal algebra is given in Section~\ref{sect:modal-alg}. In contrast to Boolean algebras, there are only few known results about computable presentations of modal algebras---here we cite~\cite{KK-12,Bazh-16}. 

It turns out that adding modality $\diamondsuit$ to the language of Boolean algebras enriches the computational content of the class: while (pure) Boolean algebras are punctually robust (informally, they are computationally `tame') \cite{kalimullin_algebraic_2017}, modal algebras are not punctually robust (i.e., some of them can exhibit a `wilder' subrecursive behavior). In Section \ref{sect:modal-alg} we prove:
\begin{restatable}{theorem}{thmModalAlgebra}\label{theo:main-modal-algebra}
 The class of modal algebras is not punctually robust.
\end{restatable}

Last but not least, we provide some elementary results about robustness of infinite binary trees (Appendix \ref{app:succ-trees}) and prefix trees (Appendix \ref{app:pref-trees}).


\section{Relational Predecessor Trees with Ordering}\label{sec:rpo}

If $P(x,y)$ holds and $x\neq r$, then we say that $x$ is a \emph{child} of $y$ (and $y$ is a parent of $x$). Notice that here we assume that the root $r$ does not have parents. If $x$ is a child of $y$ and $y$ is a child of $z$, then we say that $x$ is a \emph{grandchild} of $z$. An analogous convention holds for all representations of trees under consideration in this article.

\begin{definition}
    If $\mathcal{T}$ is a tree (r.p.o. tree or a different type) and $a$ is a node of that tree, then we define the depth of $a$ in the following way: for the root of the tree $d^{\mathcal{T}}(r)=0$ and whenever $b$ is a child of $a$, then $d^{\mathcal{T}}(b)=d^{\mathcal{T}}(a)+1$. In case of representations of trees with an empty node $e$ we also define $d^{\mathcal{T}}(e)=-1$. We also define $\mathcal{T}^{\leq n}$ to be the part of $\mathcal{T}$ consisting only of nodes with a depth less or equal $n$. 
\end{definition}

    Because of the space limit, here we only prove a specific case of Theorem~\ref{theo:predecessor-trees-with-ordering}. The full proof of Theorem~\ref{theo:predecessor-trees-with-ordering} appears in Appendix \ref{appendix:RPO-trees}.

\begin{theorem}\label{ptime}
    The class of r.p.o.~trees of unbounded depth is P-TIME robust.

\end{theorem}
\begin{proof}    We fix a computable $\mathcal{T}=(\mathbb{B},P^{\mathcal{T}},<^{\mathcal{T}},r^{\mathcal{T}})$. We carry out the construction in stages $s=0,1,\dots$. During each stage we perform a fixed number $l$ of steps of some algorithm calculating $\mathcal{T}$. We can assume that during each stage the algorithm at most once returns a true statement of the type either $P^{\mathcal{T}}(a,b)$ or $a<^{\mathcal{T}}b$ for some $a,b \in \mathbb{B}$ and that in each such case $a$ and $b$ are connected to the root of the current approximation of $\mathcal{T}$.

    We are going to construct a P-TIME copy $\mathcal{R} \cong \mathcal{T}$. At the end of each stage $s$ structures $\mathcal{T}_s$ and $\mathcal{R}_s$ are going to be our current approximations of $\mathcal{T}$ and $\mathcal{R}$ respectively. We observe that according to the construction described below both of them will be rooted at every stage but at some stages they could be not isomorphic to each other. We are also going to construct an isomorphism $\varphi: \mathcal{T} \to \mathcal{R}$.

    To ensure that the construction is P-TIME, if for some time we do not discover any new relations in $\mathcal{T}$, we can delay establishing positive cases of relations in $\mathcal{R}$ by putting new fresh strings into a reservoir set $A$. The strings in that set will not be immediately put in a fixed place in the tree $\mathcal{R}$ but we will guarantee that none of them can be in either relation with each other. Thus, while waiting to discover positive cases of relations in $\mathcal{T}$, we will keep adding negative cases of relations to $\mathcal{R}$. Since the tree has infinite depth, we are guaranteed to be able to add all the strings from the reservoir to the P-TIME tree at some point.

    We order $\mathbb{B}$ in the following way: $\alpha \sqsubset \beta$ if either $lh(\alpha)<lh(\beta)$ or $lh(\alpha)=lh(\beta)$ and additionally on the first position that they differ, $\alpha$ has $0$ and $\beta$ has $1$.

    We construct the reservoir set $A$. Initially $A = \emptyset$. At the end of each stage $s$ we take the $\sqsubset$-least sequence $\alpha$ from outside $\mathcal{R}_s$, add $\alpha$ to $A$ and declare that $\alpha$ is short.

    We want to ensure that during infinitely many stages $\mathcal{T}_s \cong \mathcal{R}_s$. We only intend to update $\mathcal{R}_s$ to satisfy this condition if we discover $a,b,c \in \mathcal{T}$ such that $P^{\mathcal{T}}(b,a)$ and $P^{\mathcal{T}}(c,b)$ and the isomorphism $\varphi(a)=\alpha$ is defined while $\varphi(b)$ and $\varphi(c)$ are not defined. Then we define $\varphi(c)=\beta$ where $\beta$ is the $\sqsubset$-least element of $A$ and we remove $\beta$ from $A$.

    For every other $d \in \mathcal{T}_s \setminus \varphi^{-1} \lbrack \mathcal{R}_s \rbrack$: $d_1,\cdots,d_j$, we take the $\sqsubset$-least sequences of length at least $s$: $\delta_1,\cdots,\delta_j$ and define $\varphi(d_i)=\delta_i$ for $1 \leq i \leq j$. We declare that all $\delta_i$ are long.

    \paragraph*{Verification}
    \begin{lemma}
    \label{LemmaCong}
        $\mathcal{R}_s \cong \mathcal{T}_s$ infinitely many times.  
    \end{lemma}

    \begin{proof}
      Suppose that the depth of $\mathcal{T}_s$ is $d_s$ at the end of some stage $s$, $\mathcal{R}_s$ is isomorphic to $\mathcal{T}_s$ and $\mathcal{R}_s$ does not change later. 
      
      Then the depth of $\mathcal{T}_t$ is at most $d_s+1$ for all $t \geq s$. This is because otherwise $\mathcal{T}_t$ would have an element $a$ of depth $d_t>d_s+1$ such that $a_0,\cdots,a_{d_t}$ is a path in $\mathcal{T}_t$ starting from the root  $a_0=r$ until $a_{d_t}=a$. In this path some $a_i$ is the last element such that $\varphi(a_i)=\alpha_i$ is defined. Since this path has not been prolonged in $\mathcal{R}_t$ beyond $\alpha_i$ we conclude that $i=d_t$ or $i=d_t-1$. Hence the depth of $\alpha_i$ is either $d_t$ or $d_t-1$. Also since $\alpha_i$ is in $\mathcal{R}_s$ we conlude that the depth of $\alpha_i$ is less or equal $d_s$. Hence $d_t-1 \leq d_s$ which is impossible.

      Since the depth of $\mathcal{T}_t$ is less or equal $d_s+1$ we obtain a contradiction with the assumption that this tree is unbounded.
    \end{proof}
The above lemma implies that $\mathcal{T} \cong \mathcal{R}$.
    \begin{lemma}
        The domain of $\mathcal{R}$ is $\mathbb{B}$.
    \end{lemma}

    \begin{proof}
        We observe that $\sqsubset$ is a well-ordering on $\mathbb{B}$. We suppose that $\alpha$ is the $\sqsubset$-least sequence not in the domain of $\mathcal{R}$.

        Then $\alpha$ was never added to $A$ because every sequence from $A$ is eventually added to $\mathcal{R}$ (since $\mathcal{R}$ gets extended infinitely many times and each time the $\sqsubset$-least sequence from $A$ is used to perform such extension). 
        The sequence $\alpha$ can only avoid getting added to $A$ if $\alpha$ is long and was earlier added to $\mathcal{R}$ but this is also a contradiction.
    \end{proof}

    \begin{lemma}
        $P^{\mathcal{R}}$ and $<^{\mathcal{R}}$ are P-TIME.
    \end{lemma}

    \begin{proof}
        We observe that if $\alpha$ and $\beta$ are in either relation from the signature, then either of them is long (maybe both of them). We assume that $\beta$ is long and that $lh(\alpha) \leq lh(\beta)=s$.

        Then to discover that $\alpha$ and $\beta$ are in a relation we perform the construction until stage $s$ and check if $\mathcal{R}_s$ confirms that they are. We observe that if $\alpha$ and $\beta$ do not enter a relation at the latest during stage $s$, then they do not later.

        We show that the above procedure is P-TIME. To perform the construction until stage $s$ we perform $l \times s$ steps of the algorithm calculating $\mathcal{T}$ and we add at most $s$ new sequences to $\mathcal{R}$, the length of each of them with a rough upper bound of $s$. Hence the algorithm is P-TIME.
    \end{proof}
Theorem \ref{ptime} is proven.\end{proof}
\section{Poset Trees}\label{sec:partially_ordered_trees}
Let $(T,\leq)$ be a poset tree. We say that elements $x , y \in T$ are adjacent, $Adj(x,y)$, if and only if $x < y \wedge \neg \exists z \in T \, (x < z < y)$, where $<$ is the strict partial order induced by $\leq$. We say $x \in T$ is a \emph{branching node} of $T$ (or $x$ branches in $T$) if it has at least two children in $T$. Such an $x$ induces a unique subtree of $T$, called a \emph{branching} and defined as $br(x,T) = \{y \in T: Adj(y,x) \} \cup T_{\geq x}$. If $x$ is a branching node, we define $|br(x,T)|$, the length of $br(x,T)$, as the length of $T_{\geq x}$. By a \emph{binary branching} we mean a tree with exactly two leaves sharing a parent.  
We say that $(T,\leq)$ is \emph{uniquely branching} if for every $n \in \mathbb N$, there exists at most one branching node $x \in T$ such that $|T_{\geq x}| = n$.
    We say that a \emph{branching} node $x \in T$ belongs to the \emph{level $n$} of $T$ if the set $T_{> x}$ contains precisely $n$ branching nodes. We denote the level $n$ of $T$ by $T[n]$ and call its members $n$-level nodes. 
Keep in mind that we number levels $0,1,\dots$. Therefore, the first level is level $0$. Level $T[n]$ may be empty but if $T[n]\neq \emptyset$ then $T[k] \neq \emptyset$, for $k < n$. 
Let $(T,\leq)$ be a tree such that its level $i$, $i \in \mathbb N$, is nonempty. We define $T_{[\leq i]}$, the \emph{$i$-level subtree of $T$}, as the least subtree of $T$ containing all nodes at levels $\leq i$ together with their children. 
Notice that $T_{[\leq i]}$ is the sum of the branchings $br(x,T)$ for all $x \in T[j]$ such that $j \leq i$.
By $r(T)$ we denote the root of the tree $T$.

\begin{lemma}\label{lemma:binary}
    Let $(T, \leq)$ be a finite poset tree in which every internal node has exactly two children and let $F \subseteq T$ be such that, at every level of $T$ except the first, at most one node is in $F$. Then there exists a leaf $y \in T$ such that $T_{\geq y} \cap F = \emptyset$. 
\end{lemma}

\begin{claimproof}
    By assumption, $r(T) \notin F$. Let $x \in T$ be such that $T_{\geq x} \cap F = \emptyset$. The node $x$ has two children which are at the same level, so one of them, denote it by $x'$, is outside $F$. Therefore, $T_{\geq x'} \cap F = \emptyset$. This way we find a leaf $y$ such that $T_{\geq y} \cap F = \emptyset$.
\end{claimproof}
    Let $T, \hat T$ be disjoint finite trees and let $z \in T$ be a leaf. $T'$ is obtained from $T$ by attaching $\hat T$ to $z$ in $T$ if $dom(T') = dom(T) \cup (dom(\hat T) - \{r(\hat T)\})$ and $x,y \in dom (T')$ satisfy $Adj_{T'}(x,y)$ iff $Adj_T(x,y) \vee Adj_{\hat T}(x,y) \vee Adj_{\hat T}(x, r(\hat T)) \wedge y = z$.

We are ready to start the proof of Theorem \ref{thmPosetTrees}. We build a computable poset tree $\mathcal T = (\mathbb N, \leq^T)$ such that for every $ i \in \mathbb N$: \begin{equation}\text{ $\mathcal P_i = (\mathbb N, \phi_i^{(2)})$ is a poset tree} \implies \mathcal T  \not\cong \mathcal P_i.\tag{$R_i$}\end{equation}
Here, $(\phi_i^{(2)})_{i\in\mathbb{N}}$ is a computable list of all binary primitive recursive functions. 

Intuitively, the strategy will work as follows. We wait (in a dovetail fashion) until $\mathcal P_i$ shows large enough fragment $P$ of itself (all nodes of $\mathcal P_i$ up to level $i$) isomorphic to the corresponding fragment of $\mathcal T$ that we construct (otherwise, we are done with $\mathcal P_i$ in the limit). We enumerate fresh elements into $\mathcal P_i$ so that the current $\mathcal P_i$ outnumbers the current approximation of $\mathcal T$. These fresh elements form subtrees of $\mathcal P_i$ attached to the leaves of $P$. We now apply the pigeonhole principle: one of the subtrees must have more elements than the corresponding subtree in $\mathcal T$---we block the growth of that part in $\mathcal T$ and make $\mathcal P_i$ non-isomorphic to $\mathcal T$.

$\mathcal T$ will be a uniquely branching poset tree with infinitely many levels . Alongside, we maintain a dynamic set $F$ of elements of $\mathcal T$ whose role is to block the growth of the tree. At odd stages the tree grows wherever $F$ permits. At even stages, the strategies for $\mathcal P_i$ monitor the relationship between $\mathcal T$ and $\mathcal P_i$ and put or withdraw elements from $F$ to satisfy $R_i$. At any given stage $s'$ we are dealing with an approximation $\mathcal P_{i,s'}$ of $\mathcal P_{i}$ defined as follows: $\mathcal P_{i,s'} = (N_{i,s'}, \leq_{i,s'})$, where $N_{i,s'} = \{0,1,\dots, s'-1\}$ and for all $x,y \in N_{i,s'}$, we have $x \leq_{i,s'} y \Leftrightarrow \phi^{(2)}_i(x,y) = 1$.


          A node $y \in \mathcal T$ is \emph{closed} (at a given stage) if there exists $x$ such that $y \leq^T x$ and $x \in F$ at that stage. A node $y$ is \emph{open} if $y$ is not closed.

\proofsubparagraph{Construction}
We proceed in stages. At the end of stage $s$, we have a finite tree $\mathcal T_s$. We set $\mathcal T = \bigcup_s \mathcal T_s$.

 At stage $s=0$, the domain of $\mathcal T_0$ is $T_0=\{0\}$ and $0 \leq^T 0$. No $R_i$ is satisfied at stage $0$, $F = \emptyset$. At each subsequent stage $s>0$, we start with a finite tree $\mathcal T_{s-1}$. 

\emph{Stage $s = 2t+1$ (expansionary).} Let $l_1, l_2, \dots, l_n$ be the open leaves of $\mathcal T_{s-1}$. Construct, for $1 \leq j \leq n$, a binary branching $b_j$ that branches at $x_j$, with the root $l_j$ and other elements fresh (i.e., the domain of $\mathcal T$ is extended by an interval $[|T_s|; b)$, for a least $b$ sufficient to construct the branchings), such that after attaching $b_j$ to $l_j$ in $\mathcal T_{s-1}$, the length of the branching in the new tree starting at $x_j$ has length $H(\mathcal T_{s-1}) + j - 1$. Declare that tree as $\mathcal T_s$ (cf. Figure \ref{fig:extendodd}). Observe that if $\mathcal T_{s-1}$ is uniquely branching, then $\mathcal T_s$ is also uniquely branching.

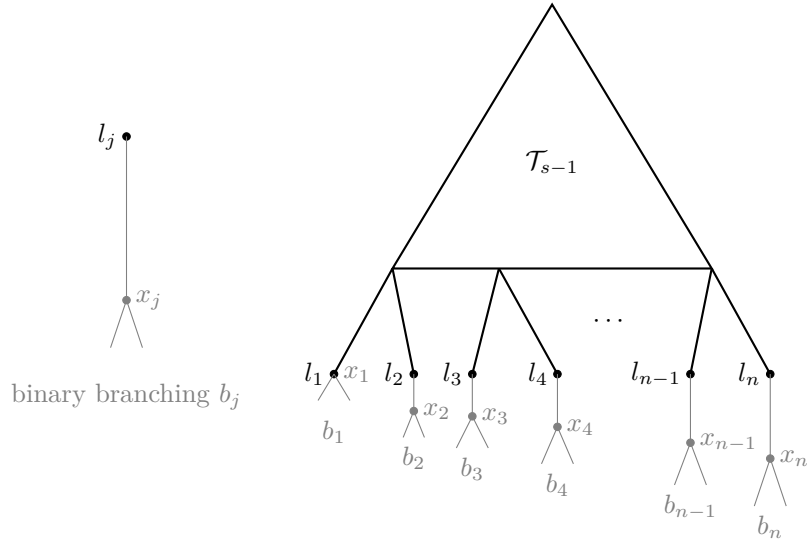
\begin{figure}
    \centering
\begin{tikzpicture}[scale=0.7]

\coordinate (left_root) at (-5, 2.5);
\coordinate (j) at ($(left_root) + (0, -1.5)$);
\coordinate (xj) at ($(j) + (0, -1.6)$);
\coordinate (left_b1) at ($(xj) + (-0.3, -0.9)$);
\coordinate (left_b2) at ($(xj) + (0.3, -0.9)$);

\filldraw[black] (left_root) circle (2pt) node[left] {$l_j$};
\draw[thin, gray] (left_root) -- (xj);
\filldraw[gray] (xj) circle (2pt) node[right] {$x_j$};
\draw[thin, gray] (xj) -- (left_b1);
\draw[thin, gray] (xj) -- (left_b2);
\node[below, gray] at ($(xj) + (0, -1.4)$) {binary branching $b_j$};

\draw[thick] (0,0) -- (3,5) -- (6,0) -- cycle;

\node at (3, 2) {$\mathcal{T}_{s-1}$};

\coordinate (A) at (0,0);
\coordinate (B) at (2,0);
\coordinate (C) at (6,0);

\coordinate (l1) at (-1.1, -2);
\coordinate (l2) at (0.4, -2);
\coordinate (l3) at (1.5, -2);
\coordinate (l4) at (3.1, -2);
\coordinate (ln1) at (5.6, -2);
\coordinate (ln) at (7.1, -2);

\draw[thick] (A) -- (l1);
\filldraw[black] (l1) circle (2pt) node[left] {$l_1$};
\filldraw[gray] (l1) circle (0pt) node[right] {$x_1$};
\draw[thin, gray] (l1) -- ++(-0.3, -0.5);
\draw[thin, gray] (l1) -- ++(0.3, -0.5);
\node[below, gray] at (-1.1, -2.7) {$b_1$};

\draw[thick] (A) -- (l2);
\filldraw[black] (l2) circle (2pt) node[left] {$l_2$};
\coordinate (x2) at ($(l2) + (0, -0.7)$);
\draw[thin, gray] (l2) -- (x2);
\filldraw[gray] (x2) circle (2pt) node[right] {$x_2$};
\draw[thin, gray] (x2) -- ++(-0.2, -0.5);
\draw[thin, gray] (x2) -- ++(0.2, -0.5);
\node[below, gray] at ($(x2) + (0, -0.5)$) {$b_2$};

\draw[thick] (B) -- (l3);
\filldraw[black] (l3) circle (2pt) node[left] {$l_3$};
\coordinate (x3) at ($(l3) + (0, -0.8)$);
\draw[thin, gray] (l3) -- (x3);
\filldraw[gray] (x3) circle (2pt) node[right] {$x_3$};
\draw[thin, gray] (x3) -- ++(-0.3, -0.6);
\draw[thin, gray] (x3) -- ++(0.3, -0.6);
\node[below, gray] at ($(x3) + (0, -0.6)$) {$b_3$};

\draw[thick] (B) -- (l4);
\filldraw[black] (l4) circle (2pt) node[left] {$l_4$};
\coordinate (x4) at ($(l4) + (0, -1.0)$);
\draw[thin, gray] (l4) -- (x4);
\filldraw[gray] (x4) circle (2pt) node[right] {$x_4$};
\draw[thin, gray] (x4) -- ++(-0.3, -0.7);
\draw[thin, gray] (x4) -- ++(0.3, -0.7);
\node[below, gray] at ($(x4) + (0, -0.7)$) {$b_4$};

\draw[thick] (C) -- (ln1);
\filldraw[black] (ln1) circle (2pt) node[left] {$l_{n-1}$};
\coordinate (xn1) at ($(ln1) + (0, -1.3)$);
\draw[thin, gray] (ln1) -- (xn1);
\filldraw[gray] (xn1) circle (2pt) node[right] {$x_{n-1}$};
\draw[thin, gray] (xn1) -- ++(-0.3, -0.8);
\draw[thin, gray] (xn1) -- ++(0.3, -0.8);
\node[below, gray] at ($(xn1) + (0, -0.8)$) {$b_{n-1}$};

\draw[thick] (C) -- (ln);
\filldraw[black] (ln) circle (2pt) node[left] {$l_n$};
\coordinate (xn) at ($(ln) + (0, -1.6)$);
\draw[thin, gray] (ln) -- (xn);
\filldraw[gray] (xn) circle (2pt) node[right] {$x_n$};
\draw[thin, gray] (xn) -- ++(-0.3, -0.9);
\draw[thin, gray] (xn) -- ++(0.3, -0.9);
\node[below, gray] at ($(xn) + (0, -0.9)$) {$b_n$};

\node at (4.1, -1) {$\cdots$};

\end{tikzpicture}
    \caption{Obtaining $\mathcal T_s$ (black and gray) from $\mathcal T_{s-1}$ (black) at odd stage $s$ by attaching binary branchings $b_j$ (left) to $l_j$ in $\mathcal T_{s-1}$, for $1 \leq j \leq n$.}
    \label{fig:extendodd}
\end{figure}

\emph{Stage $s = 2t > 0$ (non-expansionary).} Set $\mathcal T_s = \mathcal T_{s-1}$. For each $\mathcal P_i$, $0 \leq i \leq s$, apply the strategy for $\mathcal P_i$. Structurally $\mathcal T_s$ and $\mathcal T_{s-1}$ are the same, but they may differ with respect to open nodes (that is, $F$ at stage $s-1$ may be different from $F$ at stage $s$).


\proofsubparagraph{$\mathcal P_i$-strategy at stage $s = 2t$.}
Let $s' = \max(s, |T_s| + 1)$. We consider the approximation $\mathcal P_{i,s'}$ of $\mathcal P_i$. Let $s'' = \max(s-2, |T_{s-2}| + 1)$. Therefore, $\mathcal P_{i,s''}$ is the approximation of $\mathcal P_i$ that we considered at the previous non-expansionary stage (if $i \leq s''$).  We say that the strategy is ready if:
\begin{align}
    & i \leq s-2 \text{, that is, $\mathcal P_i$ was considered at the previous non-expansionary stage} \label{item:previous} \\
    & \mathcal{P}_{i,s'} \text{ is a poset tree,} \label{item:treetest} \\
    & \mathcal{T}_{s}[i+1] \neq \emptyset \text{ and } \mathcal{P}_{i,s'}[i+1] \neq \emptyset, \label{item:leveltest} \\
    & (\mathcal{P}_{i,s'})_{[\leq i + 1]}= (\mathcal{P}_{i,s''})_{[\leq i + 1]} ,  (\mathcal{T}_{s})_{[\leq i + 1]}= (\mathcal{T}_{s-2})_{[\leq i + 1]}, \text{ and}  \label{item:idtest} \\
    & (\mathcal{T}_{s})_{[\leq i + 1]} \cong (\mathcal{P}_{i,s'})_{[\leq i + 1]}. \label{item:isomtest} 
\end{align}

If the strategy is not ready, withdraw all $(i+1)$-level nodes from $F$. If the strategy is ready and some node from level $i+1$ is in $F$, we do nothing. 

The remaining case is when the strategy is ready and nodes from level $i+1$ are not in $F$. The situation is illustrated in Figure \ref{fig:strategy1}. Let $f$ be an isomorphism from $(\mathcal{T}_{s})_{[\leq i + 1]}$ to $(\mathcal{P}_{i,s'})_{[\leq i + 1]}$. Let $x_1, \dots, x_n$ be the $(i+1)$-level nodes of $(\mathcal{T}_{s})_{[\leq i + 1]}$. Each $x_j$ has two children which we denote by $l_j$ and $l_j'$. Each $l_j$ and $l_j'$ is the root of a subtree of $\mathcal T_s$ with cardinality $N_j$ and $N_j'$, respectively. In $\mathcal P_{i,s'}$, we have the corresponding subtrees with roots $f(l_j), f(l_j')$ and cardinalities $M_j, M_j'$. Since $(\mathcal P_{i,s'})_{[\leq i + 1]}\cong (T_s)_{[\leq i + 1]}$ but $\mathcal P_{i,s'}$ has more elements than $\mathcal T_s$, the surplus of elements must be placed in the subtrees just described. By the pigeonhole principle, it follows that there exists $j$, $1 \leq j \leq n$, such that $M_j + M_j' > N_j + N_j'$. Take the least such $j$ and put $x_j$ into $F$.

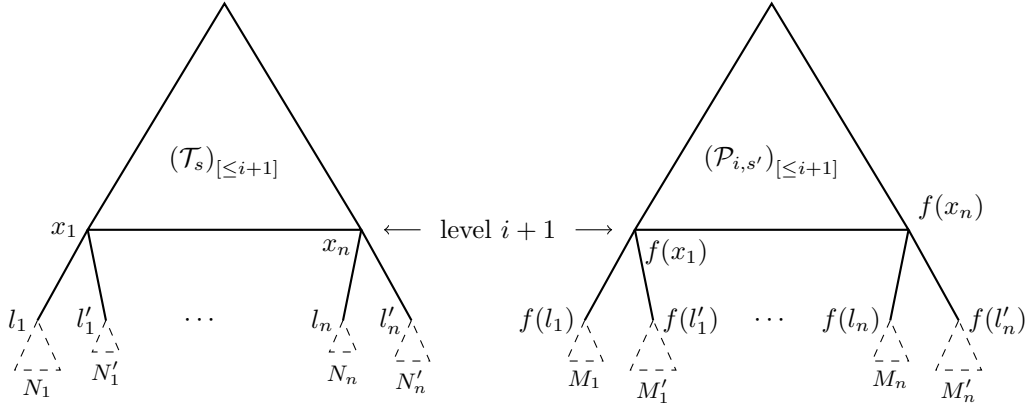
\begin{figure}
    \centering
\begin{tikzpicture}[scale=0.6]


\draw[thick] (0,0) -- (3,5) -- (6,0) -- cycle;

\node at (3, 1.5) {\(\left(\mathcal{T}_{s}\right)_{[\leq i+1]}\)};

\coordinate (A) at (0,0);
\coordinate (C) at (6,0);

\coordinate (A1) at (-1.1, -2);
\coordinate (A2) at (0.4, -2);


\coordinate (C1) at (5.6, -2);
\coordinate (C2) at (7.1, -2);

\coordinate (X1) at (0,0);
\coordinate (XN) at (6,0);

\draw[thick] (X1) -- (A1);
\draw[thick] (X1) -- (A2);


\draw[thick] (XN) -- (C1);
\draw[thick] (XN) -- (C2);

\draw[dashed] (A1) -- ++(-0.5, -1.1) coordinate (A1L) -- ++(1, 0) coordinate (A1R) -- cycle;
\draw[dashed] (A2) -- ++(-0.3, -0.7) coordinate (A2L) -- ++(0.6, 0) coordinate (A2R) -- cycle;


\draw[dashed] (C1) -- ++(-0.3, -0.7) coordinate (C1L) -- ++(0.6, 0) coordinate (C1R) -- cycle;
\draw[dashed] (C2) -- ++(-0.4, -0.9) coordinate (C2L) -- ++(0.8, 0) coordinate (C2R) -- cycle;

\node[left] at (A1) {$l_1$};
\node[left] at (A2) {$l_1'$};


\node[left] at (C1) {$l_{n}$};
\node[left] at (C2) {$l_n'$};

\node[left] at (X1) {$x_1$};
\node[below left] at (XN) {$x_n$};

\node[below, font=\footnotesize] at (-1.1, -3.1) {$N_1$};
\node[below, font=\footnotesize] at (0.4, -2.7) {$N_1'$};


\node[below, font=\footnotesize] at (5.6, -2.8) {$N_{n}$};
\node[below, font=\footnotesize] at (7.1, -2.9) {$N_n'$};

\node at (2.5, -2) {$\cdots$};


\draw[thick] (12,0) -- (15,5) -- (18,0) -- cycle;

\node at (15, 1.5) {\(\left(\mathcal{P}_{i,s'}\right)_{[\leq i+1]}\)};

\coordinate (A2) at (12,0);
\coordinate (C2) at (18,0);

\coordinate (A1_2) at (10.9, -2);
\coordinate (A2_2) at (12.4, -2);


\coordinate (C1_2) at (17.6, -2);
\coordinate (C2_2) at (19.1, -2);

\coordinate (X1_2) at (12,0);
\coordinate (XN_2) at (18,0);

\draw[thick] (X1_2) -- (A1_2);
\draw[thick] (X1_2) -- (A2_2);


\draw[thick] (XN_2) -- (C1_2);
\draw[thick] (XN_2) -- (C2_2);

\draw[dashed] (A1_2) -- ++(-0.4, -0.9) coordinate (A1L_2) -- ++(0.8, 0) coordinate (A1R_2) -- cycle;
\draw[dashed] (A2_2) -- ++(-0.5, -1.1) coordinate (A2L_2) -- ++(1, 0) coordinate (A2R_2) -- cycle;


\draw[dashed] (C1_2) -- ++(-0.4, -0.9) coordinate (C1L_2) -- ++(0.8, 0) coordinate (C1R_2) -- cycle;
\draw[dashed] (C2_2) -- ++(-0.5, -1.1) coordinate (C2L_2) -- ++(1, 0) coordinate (C2R_2) -- cycle;

\node[left] at (A1_2) {$f(l_1)$};
\node[right] at (A2_2) {$f(l_1')$};


\node[left] at (C1_2) {$f(l_{n})$};
\node[right] at (C2_2) {$f(l_n')$};

\node[below right] at (X1_2) {$f(x_1)$};
\node[above right] at (XN_2) {$f(x_n)$};

\node[below, font=\footnotesize] at (10.9, -2.9) {$M_1$};
\node[below, font=\footnotesize] at (12.4, -3.1) {$M_1'$};


\node[below, font=\footnotesize] at (17.6, -2.9) {$M_{n}$};
\node[below, font=\footnotesize] at (19.1, -3.1) {$M_n'$};

\node at (15, -2) {$\cdots$};

\node at (9, 0) {level \(i+1\)};
\draw[->] (7.3, 0) -- (6.5, 0);
\draw[->] (10.7, 0) -- (11.5, 0);

\end{tikzpicture}
    \caption{Schematic representation of $\mathcal T_s$ and $\mathcal P_{i,s'}$ from the point of view of the strategy for $\mathcal P_i$ at stage $s$. Solid lines correspond to the tree $(\mathcal T_s)_{[\leq i + 1]}$ (left) and $(\mathcal P_{i,s'})_{[\leq i + 1]}$ (right). The strategy for $\mathcal P_i$ monitors the cardinalities of the subtrees, which are represented by dashed triangles.}
    \label{fig:strategy1}
\end{figure}


\paragraph*{Verification} By construction, $\mathcal T$ is a binary poset tree and that $\mathcal T$ is uniquely branching. It is also easy to see that the domain of $\mathcal T$ is $\mathbb N$ and that $x \leq^T y$ is computable: it is sufficient to wait in an unbounded loop for a stage $s$ at which $x,y \in \mathcal T_s$ and respond how they are related in $\mathcal T_s$.

\begin{lemma}\label{lem:infinite}
$\mathcal T$ is infinite. Hence, it has infinitely many levels.
\end{lemma}
\begin{claimproof}
It is sufficient to prove that at every stage $s$, $\mathcal T_s$ has at least one open leaf (and thus, the tree is extended at each expansionary stage).
 It is easy to see that $\mathcal T_3$ has levels $0$ and $1$ (it looks like two binary branchings attached to the root). Let $s \geq 3$. Consider $\mathcal T_s'$ obtained from $\mathcal T_s$ as follows: for every branching nodes $x,y \in T_s$ such that $x <_T y$ and with no branching nodes between them, the path $\{ z \in T_s: x <_T z \leq^T y\}$ is collapsed to the single element $y$ (visually speaking, the nodes connecting two branching nodes from adjacent levels are erased). Observe that $\mathcal T_s'$ is a finite proper binary tree. By construction, $\mathcal T_s'$ and $F$ satisfy the premise of Lemma \ref{lemma:binary}. Therefore, $\mathcal T_s'$ has an open leaf. It is immediate that $\mathcal T_s$ has an open leaf as well. 
 It is easy to see that $\mathcal T$ has infinitely many branching nodes and infinitely many levels.
\end{claimproof}

\begin{lemma}
Every requirement is eventually satisfied.    
\end{lemma}
\begin{claimproof}
    Fix a requirement $R_i$ and assume $\mathcal P_i$ is a poset tree. We may additionally assume that $\mathcal P_i$ is binary and that $\mathcal P_i[i+1]$ is nonempty (that is, $\mathcal P_i$ has level $i+1$). If it does not, then $R_i$ is satisfied, because $\mathcal T$ has infinitely many levels by Lemma \ref{lem:infinite}.

Each of the properties \eqref{item:previous}, \eqref{item:treetest}, \eqref{item:leveltest} and \eqref{item:idtest} corresponds to a computable predicate of two variables, $i$ and $s$. For $\mathcal P_i$ satisfying the assumptions from the previous paragraph, the corresponding predicates hold in the limit---there is $s_0$ such that they are true when $s > s_0$:

\begin{description}
    \item[Property \eqref{item:previous}.] Immediate.
    \item[Property \eqref{item:treetest}.]  Recall that $s' = max(s, |T_s|+1)$. For large enough $s$, $r(\mathcal P_i)$ becomes a member of $\mathcal P_{i,s'}$. Once this happens, $\mathcal P_{i,s'}$ is always a poset tree.
    \item[Property \eqref{item:leveltest}.] The level $i+1$ of $\mathcal T$ is nonempty by Lemma \ref{lem:infinite}, it is present in $\mathcal T_{s}$ for large enough $s$ by construction. Above we assumed that the level $i+1$ of $\mathcal P_i$ is nonempty.
    \item[Property \eqref{item:idtest}.] For large enough $s$, $\mathcal T_{[\leq i +1]} \subseteq \mathcal T_s$ and $(\mathcal P_i)_{[\leq i+1]} \subseteq \mathcal P_{i,s'}$.

\end{description}

    Let $s_0$ be the least stage such that \eqref{item:previous}-\eqref{item:idtest} hold for all stages $s \geq s_0$. In particular, it means that starting from stage $s_0$, we have $(\mathcal P_{i,s})_{[\leq i + 1]}= (\mathcal P_i)_{[\leq i + 1]}$ and $(\mathcal T_{s})_{[\leq i + 1]}= (\mathcal T)_{[\leq i + 1]}$.

    If \eqref{item:isomtest} does not hold at stage $s_0$, then it will not hold anymore, and thus $R_i$ is satisfied. So suppose \eqref{item:isomtest} always holds starting from stage $s_0$ and let $f$ be the isomorphism from $(\mathcal{T}_{s_0})_{[\leq i + 1]}$ to $(\mathcal{P}_{i,s_0})_{[\leq i + 1]}$. Note that this isomorphism is unique because $\mathcal T$ is uniquely branching.
    
    Observe that, at stage $s_0-2$, the strategy for $\mathcal P_i$ was not ready. If it were, then $s_0' = s_0-2$ would be the least stage such that \eqref{item:previous}, \eqref{item:treetest}, \eqref{item:leveltest} and \eqref{item:idtest} hold for all even stages $s \geq s_0'$, which would contradict our choice of $s_0$. Therefore, at stage $s_0 -2$ all $(i+1)$-level nodes were withdrawn from $F$. Hence, at stage $s_0$, the strategy for $\mathcal P_i$ is ready and nodes from level $i+1$ are not in $F$. It means that we put some $x_j$ from level $i+1$ of $\mathcal T$ into $F$. At stage $s_0$, the subtree of $\mathcal T$ with the root $x_j$ has fewer elements than the corresponding subtree of $\mathcal P_{i}$ with root $f(x_j)$. And this property will hold forever because at any subsequent stage, the strategy is ready and $x_j \in F$, and in that case we do nothing. Therefore, at any subsequent expansionary stage, the subtree of $\mathcal T$ with root $x_j$ does not grow. It remains to observe that if there were an isomorphism $g$ from $\mathcal T$ to $\mathcal P_i$, it would have to map $x_j$ to $f(x_j)$, because $\mathcal T$ and $\mathcal P_i$ are isomorphic up to level $i+1$ and $\mathcal T$ is uniquely branching so $f$ has only one legitimate choice for the value $f(x_j)$.    
\end{claimproof}

We recall that a partial ordering $\mathcal{L} = (L, \leq)$ is called a join semilattice, meet semilattice, or lattice if, for any two elements $a, b$ in $\mathcal{L}$, there exists a least upper bound (denoted $a \cup b$), a greatest lower bound (denoted $a \cap b$), or both, respectively.  Join semilattices, meet semilattices, and lattices are also represented as algebraic structures, namely $(L, \cup)$, $(L, \cap)$, and $(L, \cup, \cap)$, respectively, and are characterized by appropriate axioms.  For further details on lattices, we refer the reader to the standard literature \cite{birkhoff_lattice_1940}. 

We reserve the term \emph{(semi)lattice} to refer specifically to these algebraic structures, while the term \emph{order-theoretic (semi)lattice} is used for structures of the form $(L, \leq)$. 
For (semi)lattices, the order $\leq$ and the operations $\cup$ and $\cap$ are mutually definable. Specifically, $x \leq y$ if and only if $x \cup y = y$, and $x \leq y$ if and only if $x \cap y = x$.
\corLattices*
\begin{claimproof}[Proof of Corollary \ref{cor:lattices}\eqref{item semilattices}]
    Let $\mathcal T = (\mathbb N, \subseteq)$ be the poset tree from Theorem \ref{thmPosetTrees}. Since $\subseteq$ is a partial ordering of $\mathbb N$ and the join of $x,y \in \mathbb N$ is well defined, $\mathcal T$ is an order-theoretic join semilattice and, by Theorem \ref{thmPosetTrees}, it has no punctual presentation. By defining $x \subseteq^* y \Leftrightarrow y \subseteq x$, and arguing as above, $\mathcal T^* = (\mathbb N,\subseteq^*)$ is a computable order-theoretic meet semilattice with no punctual copy.

    Let $x \cup y $ denote the join of $x$ and $y$ in $\mathcal T$. Consider the join semilattice $(\mathbb N, \cup)$. The join $x \cup y$ is a computable function of $x,y$: we run the construction from the proof of Theorem~\ref{thmPosetTrees} and wait for a stage $s$ at which both $x$ and $y$ enter $T_s$. Then we have two cases:
    \begin{enumerate}
        \item For $\subseteq$-comparable $x,y$, we return the maximum of $x,y$ with respect to $\subseteq$.
        \item For $\subseteq$-incomparable $x,y$, the construction guarantees that $T_{\supseteq x} \cup T_{\supseteq y} \subset T_s$, and thus $x \cup y$ can be computed from $ T_s$.
    \end{enumerate}
 It follows that $(\mathbb N , \cup)$ is a computable join semilattice. 
 
 To show that $(\mathbb N, \cup)$ is not punctually presentable, assume otherwise. Let $(\mathbb N, \cup')$ be a punctual copy of $(\mathbb N,\cup)$, and let $x \subseteq' y \Leftrightarrow x \cup' y = y$. Hence, the relation $x \subseteq' y$ is primitive recursive. But then $(\mathbb N, \subseteq')$ is a punctual copy of $\mathcal T$ which is impossible by Theorem \ref{thmPosetTrees}.

An analogous proof works for $(\mathbb N, \cap^*)$, the meet semilattice induced by $\mathcal T^*= (\mathbb N, \subseteq^*)$. 
\end{claimproof}
    \begin{claimproof}[Proof of Corollary \ref{cor:lattices}\eqref{item lattices}]
    Let $\mathcal L_- = (\mathbb N \setminus \{0\}, \sqsubseteq)$ be a punctual poset tree satisfying the following condition: for all $x,y \in \mathbb N \setminus \{0\}$, $x \sqsubseteq y \Leftrightarrow x -1 \subseteq y-1$. We observe that $\mathcal T \cong \mathcal L_-$ via primitive recursive isomorphism $\varphi: \mathbb N \to \mathbb N\setminus \{0\}$. The root of $\mathcal L_-$ is equal to $1$.
    
     Let $\mathcal L = (\mathbb N,\sqsubseteq)$ be such that $\mathcal L_-$ is a submodel of $\mathcal L$ and $0 \sqsubseteq x$, for all $x \in \mathbb N$ and $x \not\sqsubseteq 0$, for all $x \in \mathbb N \setminus \{0\}$. It is easy to see that $\mathcal L$ is a computable partial order. Moreover, $\mathcal L$ is a lattice. The join of $x,y \in \mathbb{N} \setminus \{0\}$ in $\mathcal L$ is the same as in $\mathcal L_-$. The join of $0,x$, where $x \in \mathbb N$, is $x$. The meet of $x,y$ such that $x \sqsubseteq y$ is $x$. The meet of $x, y$ such that $x \not\sqsubseteq y, y \not\sqsubseteq x$ is $0$. It is clear that the corresponding algebraic structure $(\mathbb N, \sqcup, \sqcap)$ where $\sqcup,\sqcap $ are the join and the meet functions of $\mathcal L$, respectively, is computable. 
    
    Observe that $\mathcal  L$ is a complemented lattice. The integer $1$ is the greatest element of $\mathcal L$ and $0$ is the least. We show that each $a \in \mathcal L$ has a complement, that is, $b$ such that $a \cup b = 1$ and $a \cap b = 0$. The construction of Theorem \ref{thmPosetTrees} guarantees that $0$ is a branching node in $\mathcal T$, and thus $1$ is a branching node in $\mathcal L$. Let $l,r$ be the two children of $1$ in $ \mathcal L$. $0$ and $1$ are complements of each other. Assume that $a \notin \{0,1\}$. Either $a \sqsubseteq l$ or $a \sqsubseteq r$. If $a \sqsubseteq l$, $a \cup r = 1$ and $a \cap r = 0$. Similarly, if $a \sqsubseteq r$, $a \cup l = 1$ and $a \cap l = 0$.

    Next, we prove that $(\mathbb N, \sqcup, \sqcap)$ is not punctually presentable. Towards a contradiction, let $\mathcal L' = (\mathbb N, \sqcup',\sqcap')$ be a punctual copy of $\mathcal L = (\mathbb N, \sqcup,\sqcap)$. Let $(\mathbb N, \sqsubseteq')$ be the order-theoretic counterpart of $\mathcal L'$, i.e., $\sqcup'$ and $\sqcap'$ are the join and meet of $\mathcal L'$, respectively. Let $0_{\mathcal L'}$ be the least element of $\mathcal L'$. It is clear that $(\mathbb N \setminus \{0_{\mathcal L'}\}, \sqsubseteq') \cong \mathcal T$. Now, given a primitive recursive bijection $f : \mathbb N \to \mathbb N \setminus \{0_{\mathcal L'}\}$, we define a punctual structure $\mathcal T' = (\mathbb N, \subseteq')$ as follows: $x \subseteq' y \Leftrightarrow f(x) \sqsubseteq' f(y)$. The primitive recursiveness of $\subseteq'$ follows from the following equivalences: $x \subseteq' y \Leftrightarrow f(x) \sqsubseteq' f(y) \Leftrightarrow f(x) \cup' f(y) = f(y)$ and the fact that $\cup'$ is primitive recursive. Hence $\mathcal T'$ is a punctual copy of $\mathcal T$ which contradicts Theorem \ref{thmPosetTrees}.
\end{claimproof}

\section{Modal Algebras}\label{sect:modal-alg}

In this section, we mainly follow notations from the monographs \cite{Goncharov-Book, Koppelberg}. Boolean algebras are viewed as algebraic structures in the signature $\{ \cup, \cap, \mathrm{C}, 0, 1 \}$. Informally speaking, one can think of the signature functions as the usual set-theoretic operations: union, intersection, and complement. In addition, $0$ is a least element, and $1$ is a greatest element.

Let $\mathcal{B}$ be a Boolean algebra. A function $f\colon \mathcal{B} \to \mathcal{B}$ is called a \emph{modality} if $f$ satisfies the following two properties:
\begin{itemize}
	\item $f(a\cup b) = f(a) \cup f(b)$ for all $a,b\in\mathcal{B}$,
	
	\item $f(0_{\mathcal{B}}) = 0_{\mathcal{B}}$.
\end{itemize}
Informally, one can view $f(x)$ as the result of applying the possibility operator $\diamondsuit = f$ to the `statement' $x$.
If $f$ is a modality, then the structure $(\mathcal{B},f)$ is called a \emph{modal algebra}. 
Here we prove the following result:

\thmModalAlgebra*


The proof idea is similar to the one for non-punctual robustness of torsion abelian groups (Theorem~3.2 in~\cite{kalimullin_algebraic_2017}). 
We build a computable modal algebra $\mathcal{A}^{\ast}$. Given a punctual `adversary' structure $\mathcal{P}_e$ (in the signature of modal algebras), we want $\mathcal{P}_e$ to show an element $c_e$ with some specific properties. Namely, either $c_e$ generates (via the function $f_{\mathcal{P}_e}$) an infinite substructure inside $\mathcal{P}_e$, or $c_e$ is a part of an $f_{\mathcal{P}_e}$-cycle of some finite size $N_e$. In the first case, we will win `automatically', since our $\mathcal{A}^{\ast}$ will not contain elements generating infinite substructures. In the second case, we try to prevent some fixed prime factor $p$ of $N_e$ from appearing in the possible sizes of the $f_{\mathcal{A}^{\ast}}$-cycles. Preventing such $p$ from appearing is quite a delicate technical task which is achieved via careful `algebraic-flavored' arrangements. In each of the two cases, we ensure that $\mathcal{A}^{\ast} \not\cong \mathcal{P}_e$. As usual, an appropriately organized priority construction allows to successfully deal with a whole computable sequence of adversaries.

We note that the construction for torsion abelian groups (from~\cite{kalimullin_algebraic_2017}) is a finite injury construction, while our construction of a modal algebra is injury-free. There are also some important differences related to the algebraic properties: for example, in the implemented construction, our $R_e$-strategy has to work with a specifically selected tuple $\bar{a} = a_0,a_1,\dots,a_M$ of witnesses (in place of just one element $c_e$ that was used in the proof idea), and this $\bar{a}$ is selected based on the Boolean algebra specifics.


Before proceeding to the formal proof of Theorem~\ref{theo:main-modal-algebra}, we give the necessary algebraic preliminaries.
For a Boolean algebra $\mathcal{B}$, its ordering $\leq_{\mathcal{B}}$ is defined in a standard way: $x\leq_{\mathcal{B}} y$ if and only if $x\cup y = y$.
An element $a\in\mathcal{B}$ is an \emph{atom} if $a$ is a minimal non-zero element in $\mathcal{B}$, i.e., $0 <_{\mathcal{B}} a$ and there is no $b$ with $0 <_{\mathcal{B}} b <_{\mathcal{B}} a$. 

By $Atom(\mathcal{B})$ we denote the set of atoms of $\mathcal{B}$.
The \emph{Fr{\'e}chet ideal} $Fr(\mathcal{B})$ is the ideal of $\mathcal{B}$ generated by the set $Atom(\mathcal{B})$. Notice that the ideal $Fr(\mathcal{B})$ contains precisely the finite sums of atoms.


Let $a$ be an element from a Boolean algebra $\mathcal{B}$. We put
	$a^0 = \mathrm{C}(a)$ and 
        $a^1  = a$.

Let $\bar a = a_0,a_1,\dots,a_n$ be a tuple of elements from $\mathcal{B}$, and let $\bar{\varepsilon} = (\varepsilon_0, \varepsilon_1, \dots, \varepsilon_n) \in \{ 0,1\}^{n+1}$. We define
	$
		\bar a^{\bar \varepsilon} = a_0^{\varepsilon_0} \cap a_1^{\varepsilon_1} \cap \dots \cap a_n^{\varepsilon_n}.
	$
The following fact is well-known (see, e.g., Exercise~6 in \S\,1.2 of~\cite{Goncharov-Book}):

\begin{lemma}\label{lemma:generating}
	Let $\mathcal{B}$ be a Boolean algebra, and let $\bar a = a_0,\dots,a_n$ be a tuple from $\mathcal{B}$. Then the subalgebra $gr_{\mathcal{B}}(\bar a)$ generated by the set $\{ a_0,\dots,a_n\}$ contains precisely the finite sums of the elements $\bar a^{\bar\varepsilon}$, $\varepsilon \in \{ 0,1\}^{n+1}$. 
	
	In addition, if $a_i \not\in \{ 0_{\mathcal{B}},1_{\mathcal{B}} \}$ and $a_i\neq a_j$ for $i\neq j$, then the finite subalgebra $gr_{\mathcal{B}}(\bar a)$ has at least $n+2$ atoms. (Notice that an atom of the subalgebra $gr_{\mathcal{B}}(\bar a)$ is not necessarily an atom of the original algebra $\mathcal{B}$.)
\end{lemma}



It is not hard to prove the following ancillary result: 

\begin{lemma}\label{lemma:modality-00}
	Let $\mathcal{B}$ be an infinite Boolean algebra, and let $g$ be an arbitrary map from the set $Atom(\mathcal{B})$ to $\mathcal{B}$. Then the function
	\[
		F_{[g]}(x) = \begin{cases}
			0_{\mathcal{B}}, & \text{if } x=0_{\mathcal{B}},\\
			g(a_0) \cup g(a_1) \cup \dots \cup g(a_n), & \text{if } x = a_0 \cup a_1 \cup \dots \cup a_n,\ a_i\in Atom(\mathcal{B}),\\
			1_{\mathcal{B}}, & \text{if } x\not\in Fr(\mathcal{B}),
		\end{cases}
	\]
	is a modality. In addition, if the algebra $\mathcal{B}$ is computable, and the sets $Atom(\mathcal{B})$ and $Fr(\mathcal{B})$ are computable, and the function $g$ is computable, then the modality $F_{[g]}$ is also computable.
\end{lemma}


\begin{definition}[forward orbit]
	Let $(\mathcal{B}, f)$ be a modal algebra. For an element $a\in\mathcal{B}$, the \emph{forward orbit} of $a$ (with respect to $f$) is the set
	$
		FOrb(a) = \{ f^{(n)}(a) : n\in \mathbb{N} \}.
	$
\end{definition}


Now we are ready to start the proof of Theorem~\ref{theo:main-modal-algebra}. By $B({\mathbb{N}})$ we denote the Boolean algebra of all finite and cofinite subsets of~$\mathbb{N}$. 
Beforehand, we choose our underlying computable Boolean algebra $\mathcal{B}$ as follows. The algebra $\mathcal{B}$ is a computable isomorphic copy of the direct product $B(\mathbb{N}) \times B(\mathbb{N})$ such that the sets $Atom(\mathcal{B})$ and $Fr(\mathcal{B})$ are computable. 

In what follows, we identify $\mathcal{B}$ with $B(\mathbb{N}) \times B(\mathbb{N})$, and we use the following notations:
\begin{itemize}
	\item $\top_0 = (1_{B({\mathbb{N}})}, 0_{B({\mathbb{N}})})$ and $\top_1 = (0_{B(\mathbb{N})}, 1_{B(\mathbb{N})})$.
	
	\item Let $\{ w_i : {i\in\mathbb{N}} \}$ be a computable list of all atoms of $B(\mathbb{N})$. We put $u_i = (w_i, 0_{B(\mathbb{N})})$ and $v_i = (0_{B(\mathbb{N})}, w_i)$. 
\end{itemize}
It is clear that $\{u_i,v_i : i\in\mathbb{N} \}$ is a computable list of all atoms of $\mathcal{B}$. In addition, we have $u_i <_{\mathcal{B}} \top_0$ and $v_i <_{\mathcal{B}} \top_1$ for all $i\in\mathbb{N}$.


By Lemma~\ref{lemma:modality-00}, it is sufficient for us to construct a computable map 
$
	g\colon Atom(\mathcal{B}) \to \mathcal{B}.
$ 
Then the desired computable modal algebra $\mathcal{A}^{\ast}$ is defined as
\begin{equation}\label{equ:A-star}
	 \mathcal{A}^{\ast} = (\mathcal{B}, F_{[g]}).
\end{equation}

\subsection{Preparations for the Construction of $g$}

Firstly, we describe the \emph{basic module} of the construction of the map $g$, we call this module:

\begin{description}
\item[Adding a $p$-cycle to the map $g$.]
Given an odd prime number $p$, we find the least $i\in\mathbb{N}$ such that the value $g(u_i)$ is not yet defined. We also choose the least $j\in\mathbb{N}$ such that $g(v_j)$ is undefined. 

Let $k = \left\lfloor \frac{p}{2}\right\rfloor$. For $0\leq m< k$, we put: 
$g(u_{i+m}) = v_{j+m}$, 
$g(v_{j+m}) = u_{i+m+1}$, and
$g(u_{i+k}) = u_i$.

We also say that the set $\{ u_{i+m} : m\leq k\} \cup \{ v_{j + \ell} : \ell <k\}$ is a \emph{$p$-cycle.} This concludes the description of the basic module.
\end{description}



We will ensure the following property of our (future) construction: 
\begin{description}
	\item[(\#)] For each odd prime $p$, we add at most one $p$-cycle to $g$.
\end{description}

Property~(\#) is enough to prove the following useful lemma about the cardinalities of forward orbits. 


\begin{restatable}{lemma}{lemmalength}\label{lemma:length}
	Suppose that the modal algebra $\mathcal{A}^{\ast}$ from Eq.~(\ref{equ:A-star}) satisfies Property~(\#). In addition, assume that there are infinitely many primes $p$ such that we have added a $p$-cycle to $g$. Let $a \in \mathcal{A}^{\ast}$ and $a \neq 0_{\mathcal{B}}$. Then the element $a$ satisfies precisely one of the following three cases:
	\begin{enumerate}
		\item $a\not\in Fr(\mathcal{B})$ and $F_{[g]}(a) = 1_{\mathcal{B}}$.
		
		\item $a \in Fr(\mathcal{B})$ and $F_{[g]}(a) = a$.
		
		\item $a\in Fr(\mathcal{B})$ and $\mathrm{card}(FOrb(a)) = q_1\cdot q_2\cdot \ldots\cdot  q_n$ for some $n\geq 1$ and some odd primes $q_i$ such that $q_i \neq q_j$ for $i\neq j$. (Note that here a $q_i$-cycle has been added to $g$ at some stage of the construction.) In addition, if $0\leq i < j < \mathrm{card}(FOrb(a))$, then $F_{[g]}^{(i)}(a) \neq F_{[g]}^{(j)}(a)$. 
	\end{enumerate}
	Furthermore, if $a\in Fr(\mathcal{B})$, $a\neq 0_{\mathcal{B}}$ and $a \leq_{\mathcal{B}} \top_k$ for some $k\in\{ 0,1\}$, then the element $a$ satisfies Case~3.
\end{restatable}
\begin{proof}
	If $a\not\in Fr(\mathcal{B})$, then Lemma~\ref{lemma:modality-00} defines $F_{[g]}(a) = 1_{\mathcal{B}}$. Thus, assume that $a\in Fr(\mathcal{B})$ and $F_{[g]}(a) \neq a$.
	
	By Property~(\#), for each prime $p$ the algebra $\mathcal{A}^{\ast}$ contains at most one $p$-cycle. We choose all prime numbers $q_1, q_2,\dots,q_n$ such that:
	\begin{itemize}
		\item there exists an atom $w\in Atom(\mathcal{B})$ such that $w\leq_{\mathcal{B}} a$ and $w$ belongs to the (unique) $q_i$-cycle; and
		
		\item there exists $w' \in Atom(\mathcal{B})$ such that $w' \not\leq_{\mathcal{B}} a$ and $w'$ belongs to the $q_i$-cycle.
	\end{itemize}
	Notice that (at least one) such prime $q_i$ exists. Indeed, if there are no such primes $q_i$, then the element $a$ satisfies the following condition:
	\begin{description}
		\item[] if an atom $w$ belongs to a $q$-cycle and $w\leq_{\mathcal{B}} a$, then \emph{every} atom from this $q$-cycle lies $\leq_{\mathcal{B}}$-below $a$.
	\end{description}
	This condition implies that $F_{[g]}(a)$ must be equal to $a$ (which contradicts our assumption).
	
	Define $M = q_1 \cdot q_2 \cdot \ldots \cdot q_n$. Then a straightforward computation shows the following: $F_{[g]}^{(M)}(a) = a$ and for all $i,j$ we have: if $0\leq i < j < M$, then $F_{[g]}^{(i)}(a) \neq F_{[g]}^{(j)}(a)$. Hence, $\mathrm{card}(FOrb(a)) = M$. 	
	We deduce that every element $a\neq 0_{\mathcal{B}}$ satisfies one of the three cases of the lemma.
	
	Now suppose that $a\in Fr(\mathcal{B}) \setminus \{ 0_{\mathcal{B}}\}$ and $a\leq_{\mathcal{B}} \top_0$. Then for some $i \in\mathbb{N}$, we have $u_i \in \mathrm{Atom}(\mathcal{B})$ and $u_i \leq_{\mathcal{B}} a$. On the other hand, for all $j \in\mathbb{N}$, we have $v_j \in Atom(\mathcal{B})$ and $v_j\not\leq_{\mathcal{B}} a$. By choosing $v_j$ from the $q$-cycle of the atom $u_i$, we deduce that $a$ must satisfy Case~3. (If $a\in Fr(\mathcal{B}) \setminus \{ 0_{\mathcal{B}}\}$ and $a\leq_{\mathcal{B}} \top_1$, then one can make a similar argument.)
	Lemma~\ref{lemma:length} is proven.
\end{proof}

\subsection{Requirements and the Construction of $g$}

Observe the following: the modal algebra $\mathcal{A}^{\ast}$ from Eq.~(\ref{equ:A-star}) has a punctual copy if and only if the structure $(\mathcal{A}^{\ast}, \top_0, \top_1)$ has a punctual copy. 
Thus, we fix a uniformly computable list $(\mathcal{P}_e)_{e\in\mathbb{N}}$ containing all punctual structures in the signature $\{ \cup,\cap, \mathrm{C},0,1,f\} \cup \{ \top_0,\top_1\}$. We satisfy the following requirements:
\begin{description}
	\item[$R_e:$] The structure $(\mathcal{A}^{\ast}, \top_0,\top_1)$ is not isomorphic to $\mathcal{P}_e$.
\end{description}
In what follows, we will abuse the notations: we identify the structures $\mathcal{A}^{\ast}$ and $(\mathcal{A}^{\ast},\top_0,\top_1)$.

At a stage $s$, we will have a finite list of \emph{active requirements}. Each active requirement $R_e$ possesses a corresponding witness $c_e \in \mathcal{P}_e$. Active requirements may be (forever) \emph{deactivated}. In addition, at a stage $s$ we will have at most one requirement $R_{e_0}$ \emph{on the alert}. 
The intended (full) life-cycle of a given requirement $R_e$ is as follows:
\begin{description}
\item[] inactive $\mapsto$ on the alert $\mapsto$ active $\mapsto$ deactivated.
\end{description}


Along the construction, we always do the following background \emph{monitoring procedure}. At a stage $s$, for each $\mathcal{P}_e$ with $e\leq s$, we consider the finite set 
$
	S_{e,s} = \{ 0_{\mathcal{P}_e}, 1_{\mathcal{P}_e}, \top_{0,\mathcal{P}_e}, \top_{1,\mathcal{P}_e}\} \cup \{ x : x\leq_{\mathbb{N}} s\}.
$ 
Assume that at the stage $s$ we have witnessed one of the following conditions:
\begin{alphaenumerate}
	\item Some of the elements from the set 
	\[
		X_{e,s} = \big\{ \bar a^{\bar \varepsilon} : \bar a = (a_1,\dots,a_n),\ 1\leq n \leq \mathrm{card}(S_{e,s}),\ a_i \in S_{e,s},\ \bar\varepsilon \in \{0,1\}^n
		\big\}
	\] 
	do not satisfy the axioms of Boolean algebras or the axioms of modal algebras. (Notice that here the set $X_{e,s}$ is the `potential' Boolean subalgebra $gr_{\mathcal{P}_e}(S_{e,s})$ generated by the set $S_{e,s}$ inside $\mathcal{P}_e$.)
	
	\item $\top_{0,\mathcal{P}_e} \cup \top_{1,\mathcal{P}_e} \neq 1_{\mathcal{P}_e}$ or $\top_{0,\mathcal{P}_e} \cap \top_{1,\mathcal{P}_e} \neq 0_{\mathcal{P}_e}$ or $f_{\mathcal{P}_e}(\top_{0,\mathcal{P}_e}) \neq 1_{\mathcal{P}_e}$ or $f_{\mathcal{P}_e}(\top_{1,\mathcal{P}_e}) \neq 1_{\mathcal{P}_e}$.
	
	\item There exist $k\in\{ 0,1\}$ and $x,y\in S_{e,s} \setminus\{ 0_{\mathcal{P}_e}, \top_{k,\mathcal{P}_e}\}$ such that $x\cup y = \top_{k,\mathcal{P}_e}$, $x\cap y = 0_{\mathcal{P}_e}$, and 
	\begin{itemize}
		\item either $f_{\mathcal{P}_e}(x) \neq 1_{\mathcal{P}_e}$ and $f_{\mathcal{P}_e}(y)\neq 1_{\mathcal{P}_e}$, or
		
		\item $f_{\mathcal{P}_e}(x) = 1_{\mathcal{P}_e}$ and $f_{\mathcal{P}_e}(y) = y$, or
		
		\item $f_{\mathcal{P}_e}(x) = 1_{\mathcal{P}_e}$ and $1_{\mathcal{P}_e} \in  \{ f^{(m)}_{\mathcal{P}_e}(y) : m\leq s \}$.
	\end{itemize}
	
	\item There exist $k\in\{ 0,1\}$ and $x,y \in S_{e,s}$ such that $x\cap y = 0_{\mathcal{P}_e}$, $x\leq_{\mathcal{P}_e}\! \top_{k,\mathcal{P}_e}$, $y\leq_{\mathcal{P}_e}\! \top_{k,\mathcal{P}_e}$, and $f_{\mathcal{P}_e}(x) = f_{\mathcal{P}_e}(y) = 1_{\mathcal{P}_e}$.
\end{alphaenumerate}

Then we declare the requirement $R_e$ deactivated. 
Indeed, in this case the structure $\mathcal{P}_e$ cannot be isomorphic to $(\mathcal{A}^{\ast}, \top_0,\top_1)$. To observe this non-isomorphism for Condition~(c) above, we recall the following fact: if $x\cup y = \top_0$, $x\cap y = 0_{\mathcal{B}}$ and $x,y\not\in \{ 0_{\mathcal{B}}, \top_{0}\}$, then precisely one of the elements $b\in \{x,y\}$ satisfies $b\not\in Fr(\mathcal{B})$ and $F_{[g]}(b) = 1_{\mathcal{B}}$. By Lemma~\ref{lemma:length}, the remaining element $a \in \{ x,y\}$ satisfies $a\in Fr(\mathcal{B})$ and $F_{[g]}(a) \neq a$. In addition, we have $1_{\mathcal{B}} \not\in FOrb(a)$.

To observe non-isomorphism for Condition~(d), we recall the following: if $x\leq_{\mathcal{B}} \top_0$ and $F_{[g]}(x) = 1_{\mathcal{B}}$, then $x\not\in Fr(\mathcal{B})$ and every $y \leq_{\mathcal{B}} \top_0 \cap \mathrm{C}(x)$ satisfies $y\in Fr(\mathcal{B})$ and $F_{[g]}(y)\neq 1_{\mathcal{B}}$.


Due to the described monitoring procedure, in the main construction given below, we may assume (without loss of generality) that every considered structure $\mathcal{P}_e$ has the following properties, for $k\in\{ 0,1\}$:
\begin{description}
	\item[(P.0)] The reduct of $\mathcal{P}_e$ to the signature $\{\cup,\cap,\mathrm{C},0,1,f\}$ is a punctual modal algebra. In addition, $\top_{0,\mathcal{P}_e} \cup \top_{1,\mathcal{P}_e} = 1_{\mathcal{P}_e}$, $\top_{0,\mathcal{P}_e} \cap \top_{1,\mathcal{P}_e} = 0_{\mathcal{P}_e}$, and $f_{\mathcal{P}_e}(\top_{0,\mathcal{P}_e}) = f_{\mathcal{P}_e}(\top_{1,\mathcal{P}_e}) = 1_{\mathcal{P}_e}$.
	
	\item[(P.1)] If $y\leq_{\mathcal{P}_e} \! \top_{k,\mathcal{P}_e}$, $y \neq 0_{\mathcal{P}_e}$ and $f_{\mathcal{P}_e}(y) \neq 1_{\mathcal{P}_e}$, then $f_{\mathcal{P}_e}(y) \neq y$ and $1_{\mathcal{P}_e} \not\in FOrb_{\mathcal{P}_e}(y)$.
	
	\item[(P.2)] If $x_1,x_2,\dots,x_n \leq_{\mathcal{P}_e} \! \top_{k,\mathcal{P}_e}$ and $x_i \cap x_j = 0_{\mathcal{P}_e}$ for $i\neq j$, then \emph{at most one} element $y\in \{ x_1,x_2,\dots,x_n\}$ satisfies $f_{\mathcal{P}_e}(y) = 1_{\mathcal{P}_e}$.
\end{description}
Now we are ready to describe the main construction. Let $(p_s)_{s\geq 1}$ be the increasing list of all odd prime numbers.


\proofsubparagraph{Construction.} At stage $0$, there are no active requirements. We declare that the requirement $R_0$ is on the alert. 


\proofsubparagraph{Stage $s > 0$.} Roughly speaking, the main goal of the stage $s$ is to decide whether to add a $p_{s}$-cycle to the map $g$. 
In addition, our actions will ensure the following property:
\begin{description} 
	\item[($\dagger$)] Suppose that by the end of the stage $s$, an active requirement $R_e$ has a witness $c_e\in \mathcal{P}_e$. Then $c_e$ satisfies one of the following two conditions:
	\begin{itemize}
		\item either the forward orbit $FOrb_{\mathcal{P}_e}(c_e)$ is infinite, or
		
		\item the set $FOrb_{\mathcal{P}_e}(c_e)$ is finite and the following implication holds:
		if $r = \mathrm{card}(FOrb_{\mathcal{P}_e}(c_e))$ has a form $r = q_1 \cdot q_2 \cdot \ldots \cdot q_n$, where $n\geq 1$ and $2 < q_1 <q_2 <\dots < q_n$ are prime numbers, then some prime $q > p_s$ must divide $r$.
	\end{itemize}
\end{description}
Intuitively speaking, the choice of such form $r = q_1 \cdot q_2 \cdot \ldots \cdot q_n$ is dictated by Lemma~\ref{lemma:length}. If the decomposition of $r$ has any other form, then Lemma~\ref{lemma:length} ensures that $\mathcal{A}^{\ast} \not\cong \mathcal{P}_e$.

If the structure $\mathcal{A}^{\ast}$ does not contain a $p_{s}$-cycle by the end of the stage $s$, then we say that the prime $p_s$ is \emph{forbidden} from entering the structure $\mathcal{A}^{\ast}$.


Our actions at the stage $s$ go as follows. If there is a requirement $R_{e_0}$ which is currently on the alert, then firstly we execute the following strategy.

\subsubsection{Strategy for $R_e$ Which Is on the Alert}

Let $\mathcal{D}_s$ be the Boolean subalgebra of $\mathcal{B}$ generated by the following elements: $\top_0$, $\top_1$, and the elements of all $q$-cycles added to $g$ at stages $t < s$. By Lemma~\ref{lemma:generating}, the algebra $\mathcal{D}_s$ is finite, and one can computably recover this structure $\mathcal{D}_s$. In addition, Lemma~\ref{lemma:modality-00} ensures that the values $F_{[g]}(x)$ are defined for all $x\in \mathcal{D}_s$.


Define $M = \mathrm{card}(\mathcal{D}_s)$. Consider the $\leq_{\mathbb{N}}$-least elements $b_0 <_{\mathbb{N}} b_1 <_{\mathbb{N}} \dots <_{\mathbb{N}} b_{M-1}$ from the structure $\mathcal{P}_e$ such that $b_i \not\in \{ 0_{\mathcal{P}_e}, 1_{\mathcal{P}_e}, \top_{0,\mathcal{P}_e}, \top_{1,\mathcal{P}_e}\}$. We define the following finite Boolean subalgebra of $\mathcal{P}_e$:\quad 
$
	\mathcal{Q}_s = gr_{\mathcal{P}_e} (\{ \top_{0,\mathcal{P}_e}, \top_{1,\mathcal{P}_e}, b_0,b_1,\dots,b_{M-1}\}).
$

By Lemma~\ref{lemma:generating}, the algebra $\mathcal{Q}_s$ has at least $M+3$ atoms. Applying Property~(P.2) of the construction, we deduce that at most two of these atoms $x$ satisfy $f_{\mathcal{P}_e}(x) = 1_{\mathcal{P}_e}$ (indeed, at most one atom below $\top_{0,\mathcal{P}_e}$, and at most one atom below $\top_{1,\mathcal{P}_e}$). Therefore, among the atoms of $\mathcal{Q}_s$ we can find $(M+1)$-many pairwise distinct elements $a_0,a_1,\dots,a_{M}$ such that $f_{\mathcal{P}_e}(a_i) \neq 1_{\mathcal{P}_e}$, $a_i\neq 0_{\mathcal{P}_e}$, and $a_i\cap a_j = 0_{\mathcal{P}_e}$ for $i\neq j$.

By Property~(P.1), we obtain that $f_{\mathcal{P}_e}(a_i) \neq a_i$ and $1_{\mathcal{P}_e} \not\in FOrb_{\mathcal{P}_e}(a_i)$. 
We define 
$
	L = \prod_{j=1}^{s} p_j. 
$
We compute the values
$
	f_{\mathcal{P}_e}^{(j)}(a_i),\ \text{for } i \leq M \text{ and } j\leq L.
$
Then one of the following five cases is satisfied:

\proofsubparagraph{Case (i.a).} There exists $a_i$ such that the function $f_{\mathcal{P}_e}$ is not injective on the set $\{ f_{\mathcal{P}_e}^{(j)}(a_i) : j\leq L\}$ (i.e., there exist $j_1 < j_2 \leq L$ such that $f_{\mathcal{P}_e}^{(j_1)}(a_i) \neq f_{\mathcal{P}_e}^{(j_2)}(a_i)$ and $f_{\mathcal{P}_e}^{(j_1+1)}(a_i) = f_{\mathcal{P}_e}^{(j_2+1)}(a_i)$).

Then we (safely) declare the requirement~$R_e$ deactivated. Indeed, by item~(3) of Lemma~\ref{lemma:length}, the structure $\mathcal{P}_e$ cannot be isomorphic to our structure $\mathcal{A}^{\ast}$. 
In all cases~(i.$X$) below, we will assume that the function $f_{\mathcal{P}_e}$ is injective on the set $\{ f_{\mathcal{P}_e}^{(j)}(a_i) : j\leq L\}$.

\proofsubparagraph{Case (i.b).} There exists $a_i$ with the following properties:
\begin{itemize}
	\item $FOrb(a_i) \subseteq \{ f_{\mathcal{P}_e}^{(j)}(a_i) : j\leq L\}$ and $N = \mathrm{card}(FOrb(a_i))$.

	\item Consider the prime decomposition $N = q_1^{\alpha_1} \cdot q_2^{\alpha_2} \cdot \ldots \cdot q_{\ell}^{\alpha_\ell}$, where $q_j \neq q_k$ for $j\neq k$ and $\alpha_j \geq 1$. For some $j\leq \ell$, either $q_j$ has already been forbidden from entering $\mathcal{A}^{\ast}$, or $q_j = 2$, or $\alpha_j \geq 2$.
\end{itemize}
Then by Lemma~\ref{lemma:length}, our structure $\mathcal{A}^{\ast}$ cannot be isomorphic to $\mathcal{P}_e$, since $\mathcal{A}^{\ast}$ does not contain elements $b\in Fr(\mathcal{B})$ with $\mathrm{card}(FOrb(b)) = N$. 
We declare the requirement~$R_e$ deactivated.


\proofsubparagraph{Case (i.c).} Neither of Cases~(i.a) and~(i.b) is satisfied, and there exists $a_i$ with the following properties:
\begin{itemize}
	\item $FOrb(a_i) \subseteq \{ f_{\mathcal{P}_e}^{(j)}(a_i) : j\leq L\}$ and $N = \mathrm{card}(FOrb(a_i))$.

	\item For some $t\geq s$, the prime $p_t$ divides $N$.
\end{itemize}
We declare the requirement $R_e$ active, and we set $c_e = a_i$. Note the following: if $t>s$, then $R_e$ will definitely satisfy Property~($\dagger$) at the end of stage $s$. 


\proofsubparagraph{Case (i.d).} Neither of Cases~(i.a)--(i.c) is satisfied, and there exists $a_i$ such that for all $j < k\leq L$, we have $f_{\mathcal{P}_e}^{(j)}(a_i) \neq f_{\mathcal{P}_e}^{(k)}(a_i)$. Notice that this condition is equivalent to the condition $\mathrm{card}(FOrb(a_i)) \geq L+1$. 
We declare the requirement $R_e$ active, and we define $c_e = a_i$.


\proofsubparagraph{Case (i.e).} Suppose that neither of Cases~(i.a)--(i.d) is satisfied. Then \emph{every} $a_i$, $i\leq M$, has the following properties:
\begin{description}
	\item[(e.1)] $FOrb(a_i) \subseteq \{ f_{\mathcal{P}_e}^{(j)}(a_i) : j\leq L\}$ and $N_i = \mathrm{card}(FOrb(a_i)) \leq L$. 
	
	\item[(e.2)] The number $N_i$ has prime decomposition $N_i = q_{i,1} \cdot \ldots\cdot q_{i,\ell_i}$, where $q_{i,j} \neq q_{i,k}$ for $j\neq k$, and $3\leq q_{i,j} < p_{s}$ and $q_{i,j}$ has not been forbidden from entering $\mathcal{A}^{\ast}$.
\end{description}
We show that in this case the structure $\mathcal{P}_e$ is not isomorphic to $\mathcal{A}^{\ast}$. 
In order to prove this, it is enough to establish the following fact: 

\begin{restatable}{claim}{claimCaseIE}\label{claim-Case-i.e}
$\mathcal{A}^{\ast}$ does not contain $(M+1)$-many pairwise disjoint elements $a$ satisfying:
\begin{equation}\label{equ:aux-001}
	\mathrm{card}(FOrb(a)) = N_i \text{ for some } N_i \text{ with Property~(e.2).} 
\end{equation}
\end{restatable}

\begin{claimproof}
Suppose that an element $a\in \mathcal{A}^{\ast}$ satisfies Eq.~(\ref{equ:aux-001}).
Consider the prime number $q_{i,1}$. By the proof of Lemma~\ref{lemma:length}, there exists an atom $a' \in Atom(\mathcal{B})$ such that $a' \leq_{\mathcal{B}} a$ and $a'$ belongs to a $q_{i,1}$-cycle. By the definition of the finite algebra $\mathcal{D}_s$, we have $a' \in \mathcal{D}_s$. We define $\theta(a) = a'$.
We notice the following fact: if $a\cap b = 0_{\mathcal{B}}$, then $\theta(a) \neq \theta(b)$. Indeed, if $\theta(a) = \theta(b)$, then $0_{\mathcal{B}} \neq \theta(a) \leq_{\mathcal{B}} a\cap b$.

Now, towards a contradiction, assume that $\mathcal{A}^{\ast}$ contains $(M+1)$-many pairwise disjoint elements $a$ satisfying Eq.~(\ref{equ:aux-001}). Then the algebra $\mathcal{D}_s$ contains at least $(M+1)$-many pairwise distinct elements $\theta(a)$. This contradicts the choice of $M  = \mathrm{card}(\mathcal{D}_s)$.
\end{claimproof}

Indeed, recall that $a_i\cap a_j = 0_{\mathcal{P}_e}$ for $i\neq j$. Hence, the structure $\mathcal{P}_e$ \emph{does contain} $(M+1)$-many disjoint elements $a_0,\dots,a_{M}$ satisfying Eq.~(\ref{equ:aux-001}). Thus, Claim~\ref{claim-Case-i.e} implies that $\mathcal{P}_e \not\cong \mathcal{A}^{\ast}$. 
Therefore, in Case~(i.e), we safely declare the requirement $R_e$ deactivated. 


This concludes the description of the strategy for $R_e = R_{e_0}$ which is on the alert. 
After executing this strategy, one-by-one, we execute strategies for the currently active requirements $R_{e'}$ (see below). Notice that in Cases~(i.c) and~(i.d) above, this execution also includes our requirement $R_{e_0}$ ($R_{e_0}$ had been on the alert, but then it moved to becoming active).

\subsubsection{Strategy for an Active Requirement $R_e$}

As usual, here we assume the following: if the requirement $R_e$ was active at the end of the previous stage $s-1$, then it satisfied Property~($\dagger$) at that moment. 
Consider the witness $c_e$ for $R_e$. 
Recall that
\begin{equation}  \label{equ:length-agreement}
	L = \prod_{j=1}^{s} p_j.
\end{equation}
We compute the values $f_{\mathcal{P}_e}^{(i)}(c_e)$, for all $i\leq L$. One of the following two cases is satisfied:


\proofsubparagraph{Case~(ii.a).} Suppose that $f_{\mathcal{P}_e}^{(j)}(c_e) = f_{\mathcal{P}_e}^{(k)}(c_e)$ for some $j<k\leq L$. 
Then we have $FOrb_{\mathcal{P}_e}(c_e) \subseteq \{ f_{\mathcal{P}_e}^{(i)}(c_e) : i\leq L\}$. In particular, the set $FOrb_{\mathcal{P}_e}(c_e)$ is finite. For $N = \mathrm{card}(FOrb_{\mathcal{P}_e}(c_e))$, we compute its prime decomposition
$
	N = q_1^{\alpha_1} \cdot q_2^{\alpha_2} \cdot \ldots \cdot q_{\ell}^{\alpha_\ell}.
$

Similarly to Case~(i.b), if for some $i\leq \ell$ either $q_i$ was already forbidden, or $q_i = 2$, or $\alpha_i\geq 2$, then $\mathcal{P}_e \not\cong \mathcal{A}^{\ast}$. Thus, we declare such $R_e$ deactivated. Hence, in what follows we may assume that 
$
	N = q_1 \cdot q_2 \cdot \ldots \cdot q_{\ell},
$
where $2 < q_1 < q_2 < \dots < q_{\ell}$.

If some $q_i$ equals $p_s$, then we \emph{forbid} adding a $p_s$-cycle to the map $g$. Since $\mathcal{A}^{\ast}$ will not have $p_s$-cycles, by Lemma~\ref{lemma:length}, we deduce that $\mathcal{P}_e$ is not isomorphic to $\mathcal{A}^{\ast}$. We safely declare the requirement $R_e$ deactivated.

If $p_s\not\in \{ q_i : i\leq \ell\}$, then $R_e$ stays active. Here we claim that some $q_i$ \emph{must be equal} to $p_t$ for some $t >s$. Indeed, if $R_e$ was already active at the stage $s-1$, then this fact is guaranteed by the `$(s-1)$-version' of Property~($\dagger$). Otherwise, $R_e$ moved from being on the alert to being active at the stage $s$. But then this move was triggered by Case~(i.c), and some $p_t$ with $t>s$ must divide $N$. 
We observe that in Case~(ii.a), the requirement $R_e$ will satisfy Property~($\dagger$) at the end of the stage $s$.


\proofsubparagraph{Case~(ii.b).} Otherwise, $f_{\mathcal{P}_e}^{(j)}(c_e) \neq f_{\mathcal{P}_e}^{(k)}(c_e)$ for all $j<k\leq L$. Then the requirement $R_e$ stays active. 
Observe that here $N = \mathrm{card}(FOrb_{\mathcal{P}_e}(c_e)) \geq L+1$.

In Case~(ii.b) we need to show that the requirement $R_e$ will still satisfy Property~($\dagger$) by the end of the stage $s$. Towards a contradiction, assume that Property~($\dagger$) fails. Then the set $FOrb_{\mathcal{P}_e}(c_e)$ is finite, and we have $N = q_1 \cdot q_2\cdot \ldots \cdot q_{\ell}$, where the primes $q_i$ satisfy $2 < q_1 < q_2 <\dots <q_\ell \leq p_s$. By Eq.~(\ref{equ:length-agreement}), we obtain that
$
	N = q_{\ell} \cdot q_{\ell-1} \cdot \ldots \cdot  q_2\cdot q_1 \leq p_s \cdot p_{s-1} \cdot \ldots \cdot p_{s-(\ell-2)} \cdot p_{s-(\ell-1)} \leq L.
$
This contradicts the fact that $N > L$. 

We deduce that in each of the two cases, an active requirement $R_e$ will satisfy Property~($\dagger$) by the end of the stage $s$. This concludes the description of the strategy for an active $R_e$.


If by the end of the stage $s$, no strategy has forbidden to add a $p_s$-cycle to $g$, then we proceed as follows:
\begin{itemize}
	\item Add a $p_s$-cycle to the map $g$.
	
	\item Find the least $i$ such that the requirement $R_i$ has never been on the alert before. Declare this $R_i$ being on the alert.
\end{itemize}



\subsection{Verification}

\begin{lemma}\label{lem:01-totality}
	For every $a\in Atom(\mathcal{B})$, the value $g(a)$ is eventually defined. Consequently, the structure $\mathcal{A}^{\ast}$ from Eq.~(\ref{equ:A-star}) is a well-defined computable modal algebra.
\end{lemma}
\begin{proof}
	It is sufficient to show that our construction adds a $p_s$-cycle for infinitely many $s\geq 1$. Towards a contradiction, assume that for every $s> s_0$, the structure $\mathcal{A}^{\ast}$ never gets a $p_s$-cycle. Then for every $s>s_0$, the requirement $R_s$ is never declared on the alert. Consequently, such $R_s$ never becomes active.

	Choose a large enough stage $s^{\ast} > s_0$ with the following property: if a requirement $R_e$, where $e\leq s_0$, is eventually declared deactivated, then $R_e$ was deactivated \emph{before} the stage $s^{\ast}$. Since the $p_{s^{\ast}}$-cycle is forever forbidden, this forbiddance was triggered by the following action: at the stage $s^{\ast}$, some currently active requirement $R_i$, where $i\leq s_0$, forbade $p_{s^{\ast}}$, and after that this $R_i$ was deactivated. But this contradicts the choice of the stage $s^{\ast}$. 
	We deduce that $g(a)$ is defined for all atoms $a$ from $\mathcal{B}$. Lemma~\ref{lem:01-totality} is proved.
\end{proof}


\begin{lemma}\label{lem:02-reqs-satisfied}
	Every requirement $R_e$ is satisfied.
\end{lemma}
\begin{proof}
	Since our construction adds a $p_s$-cycle for infinitely many $s$ (as discussed in Lemma~\ref{lem:01-totality}), every requirement $R_e$ is eventually declared being on the alert. If a requirement $R_e$ is eventually deactivated, then (as discussed in detail in the construction description) we have $\mathcal{P}_e \not\cong \mathcal{A}^{\ast}$ and $R_e$ is satisfied.
	
	Suppose that a requirement $R_e$ is never deactivated. Then there exists a stage $s_0$ such that $R_e$ is active at every stage $s \geq s_0$. Consider the corresponding witness $c_e\in\mathcal{P}_e$. 	
	If the forward orbit $FOrb_{\mathcal{P}_e}(c_e)$ is infinite, then by Lemma~\ref{lemma:length}, $\mathcal{P}_e$ is not isomorphic to $\mathcal{A}^{\ast}$. Therefore, we may assume that the set $FOrb_{\mathcal{P}_e}(c_e)$ is finite. 
	
	Since $N = \mathrm{card}(FOrb_{\mathcal{P}_e}(c_e)) <\omega$, there exists a large enough stage $s^{\ast} \geq s_0$ such that for every $s\geq s^{\ast}$, the active requirement $R_e$ satisfies Case~(ii.a) at the stage $s$. The requirement $R_e$ is never deactivated, hence, we have
	$
		N = q_1 \cdot q_2 \cdot \ldots \cdot q_{\ell}
	$
	for some $\ell \geq 1$ and some primes $2 < q_1 < q_2 <\dots < q_{\ell}$. But then Property~($\dagger$) of the construction implies that for every $s \geq s^{\ast}$, there exists a prime $q > p_s$ such that $q$ divides $N$. Hence, $N$ has infinitely many divisors, which gives a contradiction.	We conclude that for every $e$, the structure $\mathcal{A}^{\ast}$ is not isomorphic to $\mathcal{P}_e$.
\end{proof}



\bibliography{ref}

\appendix

\section{Binary successor trees}\label{app:succ-trees}

A binary tree is a fundamental concept in computer science. Its defining properties are as follows: there is a distinguished node $r$, called the root, and each node has at most two children, referred to as the \emph{left} and \emph{right} child. We adopt a representation that provides left and right successor functions, $S^1$ and $S^2$, which map each node to its left or right child, respectively, or to the “empty node” $e$ which indicates the absence of a value. This representation resembles a common memory model for binary trees (see, e.g., \cite{knuth_art_1968}, p. 316).
\begin{definition}[binary successor tree]
  $\mathcal{T}=(T,S^1,S^2,e,r)$ is a binary successor tree iff $S^1$ and $S^2$ are both unary functions satisfying:
  \begin{enumerate}
      \item $S^1$ and $S^2$ do not have any cycles and are injective on $\mathbb{N} \setminus (S^i)^{-1}(e)$,
      \item $ran(S^1) \cap ran(S^2) = \{e\}$,
      \item $ran(S^1) \cup ran(S^2) = \mathbb{N}\setminus \{r\}$,
      \item $S^1(e)=S^2(e)=e$.
  \end{enumerate}
\end{definition}

We observe that punctual presentability of a tree does not depend on including both constants $e$ and $r$ in the signature. Hence any results of this sort remain valid even if either constant or both of them are removed from the signature.

\begin{definition}
   We say that a successor tree is $n$-full if every node of depth $n-1$ has both non-empty successors.
\end{definition}


\begin{theorem}
    The class of successor trees is not punctually robust.
\end{theorem}

\begin{proof}
    We fix an enumeration of all punctual structures over the same signature as successor trees $\mathcal{T}_n=(\mathbb{N},S^1_n,S^2_n,e_n,r_n)$ and we want to construct a tree $\mathcal{T}$ satisfying all requirements:
    \[
    (\mathcal{R}_s) \; \mathcal{T} \not \cong \mathcal{T}_s. 
    \]
    The construction will be carried out in infinitely many stages, at each stage $s$ the requirement $\mathcal{R}_s$ is going to be satisfied. At stage $0$ define the empty node to be number $0$ and the root to be number $1$.

    At the beginning of stage $s>0$ all requirements $\mathcal{R}_0, \dots, \mathcal{R}_{s-1}$ have already been satisfied and $\mathcal{T}^{\leq s-1}$ has already been constructed. We calculate all of $\mathcal{T}_s^{\leq s}$. There are several possibilities.

    If $\mathcal{T}_s^{\leq s}$ is not a successor tree, then the requirement is already satisfied and we can move on to the next stage. Similarly if $\mathcal{T}^{\leq s-1} \not \cong \mathcal{T}_s^{\leq s-1}$. Now we need to consider the case when both these trees are isomorphic. We want to ensure that they will somehow be distinguished by nodes of depth $s$. We assume that both of them have $n$ nodes of depth $s-1$: $a_1, \dots, a_{n}$.
    
    If $\mathcal{T}_s^{\leq s}$ is not $s$-full, then we take $2n$ least fresh natural numbers $b_1, \dots, b_{2n}$, add them to the domain of $\mathcal{T}$ and we extend functions $S^1$ and $S^2$, defining $S^1(a_i)=b_i$ and $S^2(a_i)=b_{i+n}$, $1 \leq i \leq n$.

    If $\mathcal{T}_s^{\leq s}$ is $s$-full, then we take $2n-1$ least unused natural numbers $b_1, \dots, b_{2n-1}$, add them to the domain of $\mathcal{T}$ and we extend functions $S^1$ and $S^2$, defining $S^1(a_i)=b_i$ (if $1 \leq i \leq n$), $S^2(a_i)=b_{i+n}$ (when $1 \leq i \leq n-1$) and $S^2(a_n)=0$ (we recall that $0$ represents an empty node). Thus we ensured that $\mathcal{T}^{\leq s}$ is $s$-full iff $\mathcal{T}_s^{\leq s}$ is not $s$-full. Hence these trees are not isomorphic and $\mathcal{R}_s$ is satisfied.

    We observe that whenever $\mathcal{T}$ has $n$ nodes of a certain depth $s$, then there are either $2n$ or $2n-1$ nodes of depth $s+1$. Hence the construction never runs out of nodes to prolong.
\end{proof}

\section{Prefix trees}\label{app:pref-trees}
A prefix tree is a familiar object in computability theory and in descriptive set theory. In general, $T$ is a prefix tree over a set $X$ if $T$ is a subset of $X^{<\omega}$, the set of all finite sequences of elements of $X$, and satisfies the following condition: for every $\vec a, \vec b \in X^{<\omega}$, if $\vec a \in T$ and $\vec b$ is a prefix of $\vec a$, then $\vec b \in T$.

The concept of a prefix tree, as defined above, is not formulated in model-theoretic terms, which we require for assessing punctual robustness. Therefore, we need a precise specification of a domain and a signature, along with defining characteristics that determine which corresponding models qualify as `prefix trees.'

Below, we adopt a definition that appears most natural. However, this definition involves an infinite signature. Unfortunately, the assumed definition of a punctual structure (see Definition \ref{def:punctual structure}) does not specify how to handle cases where the signature is infinite. Thus, we must take a step back.

So far, no generally accepted definition of a punctual structure with an infinite signature exists. Various approaches can be taken to introduce such a notion (for a corresponding framework in polynomial-time algebra, see, e.g., \cite{cenzer_feasibly}). Here, we adopt a particularly strong definition---likely the strongest among all plausible variants.

\begin{definition}
    A structure $\mathcal{A}=(A, (R_i^{\mathcal{A}})_{i \in I},(f_j^{\mathcal{A}})_{j \in J},(c_k^{\mathcal{A}})_{k \in K})$ is punctual if $A$ is equal to $\mathbb N$ or to a finite initial segment of $\mathbb N$, all relations, functions, and constants from the signature are uniformly primitive recursive, and the function assigning to the index of a relation or of a function the arity is primitive recursive.
\end{definition}

Other variants of the above definition could weaken in various ways the requirement that the arity function is primitive recursive, asking only that this function is recursive or that for any fixed number we can decide in a primitive recursive way if a certain fixed function or relation has that many arguments.

\begin{definition}
    $\mathcal{T}=(T,R_1^{\mathcal{T}},R_2^{\mathcal{T}},R_3^{\mathcal{T}},\dots;r^{\mathcal{T}})$ is a prefix tree if each $R_n^{\mathcal{T}}$ is an $n$-ary relation, $r^{\mathcal{T}} \in T$ and the following are satisfied:
    \begin{enumerate}
        \item $R_1^{\mathcal{T}}=\{r^{\mathcal{T}}\}$,
        \item if $R_n^{\mathcal{T}}(a_1,\dots,a_n)$ and $1 \leq i \leq n$, then $R_i^{\mathcal{T}}(a_1,\dots,a_i)$.
    \end{enumerate}
    We say that a prefix tree is injective if whenever $\vec{a}=(a_1,\dots,a_n,c)$ and $\vec{b}=(b_1,\dots,b_k,c)$ and $R_{n+1}^{\mathcal{T}}(\vec{a})$ and $R_{k+1}^{\mathcal{T}}(\vec{b})$, then $n=k$ and $a_1=b_1, \dots, a_n=b_n$.
\end{definition}

The interpretation of the above definition is such that $R_n^{\mathcal{T}}$ is the set of all paths of length $n$ in the prefix tree $\mathcal{T}$ which start at the root $r^{\mathcal{T}}$. This is in line with a common practice of defining trees to be sets of strings containing all prefixes of every string included. We  slightly deviate from that convention in that we do not include the empty string.

\begin{definition}
    If $\vec{a}=(a_1,\dots,a_n)$, $\vec{b}=(b_1,\dots,b_k)$, $k<n$ and $a_i=b_i$ whenever $1 \leq i \leq k$, then we say that $\vec{a}$ extends $\vec{b}$ by $n-k$.
\end{definition}

\begin{theorem}\label{pref}
    The class of injective prefix trees is punctually robust.
\end{theorem}

\begin{proof}
    We fix a computable injective prefix tree $\mathcal{T}$. We are going to construct a punctual $\mathcal{R} \cong \mathcal{T}$. We assume that $\mathcal{T}$ has an infinite branch. Otherwise the tree has an infinitely branching node and the proof is standard.
    The general idea is that we enumerate the children of that node while waiting to copy the rest of the tree. Also please check the proof of Proposition~3 in \cite{kalocinski_punctual_2024}.

    
    The construction will be carried out in infinitely many stages and at the end of each stage $s$ we will have determined whether any fixed $s$-bounded $(a_1,\dots,a_n) \in R_n^{\mathcal{T}}$. Here we say that $(a_1,\dots,a_n)$ is $s$-bounded if $a_1,\dots,a_n \leq s$. At the end of stage $s-1$ we have already recovered some $\mathcal{T}_{s-1}$ (a fragment of $\mathcal{T})$ and constructed $\mathcal{R}_{s-1}$---the image of $\mathcal{T}_{s-1}$ under the isomorphism $\varphi$.

   During the entire construction we run an algorithm calculating $\mathcal{T}$ and build an isomorphism $\varphi: \mathcal{T} \to \mathcal{R}$. We are also going to build $W$---a set of all paths in $\mathcal{T}$ which have already been discovered but are still waiting to get implemented into $\mathcal{R}$. Initially $W=\emptyset$.
   
   At stage $s$ we check if we obtained any new information about $\mathcal{T}$. If not, then we declare that all $s$-bounded $(a_1,\dots,a_n)$ do not belong to $R_n^{\mathcal{R}}$ unless we already declared earlier that they do.

   Now we assume that at stage $s$ we discovered that some $\vec{a}=(a_1,\dots,a_n) \in R_n^{\mathcal{T}}$ and that $\vec{b}=(b_1,\dots,b_m)$ is the longest branch in $\mathcal{T}_{s-1}$ extended by $\vec{a}$. If $n < m+2$, then we add the path $(a_1,\dots,a_n)$ to $W$.

   If $n \geq m+2$, then we extend $\mathcal{R}$ with all the paths in $W$ and then we declare that $W:=\emptyset$. Below we describe in more detail how this is done.

   Whenever $W$ contains a path $\vec{a}=(a_1,\dots,a_n)$ and $\vec{b}=(b_1,\dots,b_k) \in R_k^{\mathcal{T}}$ is the longest branch already copied to $\mathcal{R}$ extended by $\vec{a}$, then we take $t:=n-k-1$ least unused numbers $c_1,\dots,c_t$ and the least unused number $c_{t+1} \geq s$ and we define $\varphi(a_{i+k})=c_i$, $1 \leq i \leq t$ and $\varphi(a_n)=c_{t+1}$. We observe that $\vec{C}=\varphi(\vec{a})$ contains an element $c_{t+1} \geq s$ and hence has not been banned from entering $R_n^{\mathcal{R}}$ at any earlier stage. We also observe that since $\mathcal{T}$ contains an infinite path, there will always be a stage when some already existing branch is extended by at least two, and hence every path from $W$ will at some point be incorporated into $\mathcal{R}$.
\end{proof}

\begin{theorem}
    The class of prefix trees is not punctually robust.
\end{theorem}

\begin{proof}
  First we define an auxiliary notion of a binary prefix tree. The definition is the same as in the case of prefix trees on an infinite set $T$ with the only difference being that the domain is the set $A=\{0,1\}$.
  
  We observe that any binary prefix tree is punctually presentable if and only if that tree is already punctual and hence that not all of them are punctually presentable.

  We fix a non-punctually presentable binary prefix tree $\mathcal{T}_b= (A,R_1^{\mathcal{T}_b},R_2^{\mathcal{T}_b},\dots, r^{\mathcal{T}_b} )$. We assume that the root $r^{\mathcal{T}_b}$ is $1$. We define $\mathcal{T}=(\mathbb{N},R_1^{\mathcal{T}},R_2^{\mathcal{T}},\dots,r^{\mathcal{T}})$ in the following way. For each $i\geq 1$, $R_i^{\mathcal{T}}=R_i^{\mathcal{T}_b} \cup \{(1,2,\dots, i)\}$ and $r^{\mathcal{T}}=1$. We observe that if $\mathcal{T}$ had a punctual presentation, then so would $\mathcal{T}_b$. 
\end{proof}

The definition below works with both finite and infinite signatures and is very similar to that in the punctual case.

\begin{definition}
    A structure $\mathcal{A}=(\mathbb{B},(R_i^{\mathcal{A}})_{i \in I},(f_j^{\mathcal{A}})_{j \in J},(c_k^{\mathcal{A}})_{k \in K})$ (where each of the sets $I,J,K$ is either the set of all positive natural numbers or a finite initial segment thereof) is fully P-TIME if all relations, functions, and constants from the signature are uniformly P-TIME and the function assigning to the index of a relation or a function the arity is P-TIME.
\end{definition}

We understand the above definition in the following way. Initially we have a finite sequence of elements of the alphabet $\{0,1,a,b,c,d\}$ on the tape. The input begins with a sequence of letters $a$, $b$ or $c$, where the number of occurrences of the letter symbolises the index of, respectively, a relation, a function or a constant under consideration. This is followed by the list of arguments of a relation or a function or the element of the domain to be represented by the constant. When dealing with non-unary relations or functions, the arguments are separated with the letter $d$. Hence, if we want to check if $(001,010) \in R_3^{\mathcal{A}}$, we are going to work on the sequence $aaa001d010$ as the input.

We require that there is a polynomial function $P$ such that for any input (defined as above) of length $n$, the algorithm halts in at most $P(n)$ steps. This is a very strong definition, and some weaker variants are possible. A natural weakening of the definition would be to have individual polynomial bounds on different relations and functions from the signature (preferably retrievable in a primitive recursive way). A variant of a definition was considered in \cite{cenzer_feasibly}.

\begin{theorem}
    The class of injective prefix trees of unbounded depth is P-TIME robust.
\end{theorem}

\begin{proof}
    We only sketch the proof. The general idea of the construction is very similar to that of Theorem \ref{pref}. To ensure that the construction works in polynomial time we introduce and use the distinction between short and long sequences from the proof of Theorem \ref{ptime}.

    We copy a fixed tree $\mathcal{T}$ into some P-TIME tree $\mathcal{R}$ while also constructing an isomorphism $\varphi$ between them. Whenever some path in $\mathcal{T}$ had already been copied to $\mathcal{R}$, the current leaf in the tree under construction is $\alpha$ such that $\varphi(a)=\alpha$ for some $a \in \mathcal{T}$ and we have discovered $b$ and $c$ such that $b$ is the child of $a$ and $c$ is the child of $b$, then we prolong the isomorphism. We set $\varphi(b)$ to be the least unused short sequence currently in the reservoir and we use long sequences to define $\varphi(k)$ for all other elements discovered in $\mathcal{T}$.       
\end{proof}

\section{Proof of Theorem~\ref{theo:predecessor-trees-with-ordering}}\label{appendix:RPO-trees}

\theoRPOtrees*

\begin{proof}
Firstly, we give a detailed proof for the case of punctual robustness.

We fix an infinite computable r.p.o.~tree $\mathcal{T}$. We choose a \emph{primitive recursive} approximation $(\mathcal{T}_s)_{s\in\mathbb{N}}$ of the tree $\mathcal{T}$ satisfying the following properties:
    \begin{itemize}
        \item every $\mathcal{T}_s$ is a finite tree and $\mathcal{T} = \bigcup_{s\in\mathbb{N}} \mathcal{T}_s$;

        \item $r\in \mathcal{T}_s \subseteq \mathcal{T}_{s+1}$ and $\mathrm{card}(\mathcal{T}_{s+1} \setminus \mathcal{T}_s) \leq 1$ for all $s\in\mathbb{N}$.
    \end{itemize}
    Intuitively speaking, $(\mathcal{T}_s)_{s\in\mathbb{N}}$ is a punctually non-decreasing sequence of finite trees such that for each $s$, either $\mathcal{T}_{s+1}$ is equal to $\mathcal{T}_s$, or $\mathcal{T}_{s+1}$ is obtained by appending only one new node to~$\mathcal{T}_s$.

    We want to construct a punctual tree $\mathcal{R} \cong \mathcal{T}$. If $\mathcal{T}$ has an infinite depth, then 
the existence of such a tree $\mathcal{R}$ follows from Theorem~\ref{ptime}.

Otherwise, the depth of the tree $\mathcal{T}$ is finite. Since $\mathcal{T}$ itself is infinite, there exists a node $a\in \mathcal{T}$ which has infinitely many children. Here our construction is split into two cases.

\proofsubparagraph{Case A.} Suppose that there exists a node $a\in\mathcal{T}$ with the following property: 
\begin{description}
	\item[($\ddagger$)] $a$ has infinitely many children $x$ such that $x$ also has its own child. 
\end{description}
In particular, this implies that $a$ has infinitely many grandchildren which are pairwise incomparable with respect to the ordering~$<_{\mathcal{T}}$.
Without loss of generality, we assume that $a \in \mathcal{T}_0$.

The construction proceeds in stages. At a stage $s$, we build a finite structure $\mathcal{R}_s$. We also construct a partial isomorphic embedding $\psi_s :\,\subseteq \mathcal{R}_s \to \mathcal{T}_s$. In the limit, the map $\psi = \bigcup_{s\in\mathbb{N}} \psi_s$ will be a computable isomorphism from the tree $\mathcal{R} = \bigcup_{s\in\mathbb{N}} \mathcal{R}_s$ onto $\mathcal{T}$.

The intuition behind the construction is as follows. In order to ensure the punctuality of the constructed structure $\mathcal{R}$, we `quickly' produce some (intended) grandchildren $d_0,d_1,d_2,\ldots$ of the node $\psi^{-1}(a)$ which are pairwise $<_{\mathcal{R}}$-incomparable. Inside $\mathcal{R}$, we \emph{do not connect} a node $d_i$ to (the component of) the node $\psi^{-1}(a)$ until we see an appropriate isomorphic image $z = \psi(d_i)$ in $\mathcal{T}_s$. (In particular, at some stages $s$ the structure $\mathcal{R}_s$ will be a finite forest that is not a tree.) Only after we find such an image $z$, we add a fresh node $e_i$ to $\mathcal{R}$ and declare $P(d_i, e_i) \, \& \, P(e_i, \psi^{-1}(a))$ inside $\mathcal{R}$ (i.e., we make $d_i$ a `real' grandchild of $\psi^{-1}(a)$).

Our construction of $\mathcal{R}$ proceeds as follows. At stage $0$ we put $\mathcal{R}_0=\mathcal{T}_0$ and $\psi_0 = \mathrm{id}_{\mathcal{T}_0}$.

In addition, at the end of each stage $s$, we choose the least natural number $x$ which currently does not belong to $\mathcal{R}_s$. We define $d_s = x$ and add $d_s$ to $\mathcal{R}_s$. For each $y\in \mathcal{R}_s\setminus \{ d_s\}$, we declare that $y$ is $<_{\mathcal{R}}$-incomparable with $d_s$ and that $\neg P(y,d_s) \, \& \, \neg P(d_s,y)$.

At a stage $s+1$, we find the least $i\leq s$ such that $\psi_s(d_i)$ is undefined. 
Suppose that the tree $\mathcal{T}_{s+1}$ contains nodes $y,z\not\in \mathrm{range}(\psi_s)$ such that $\mathcal{T} \models P(z,y)\& P(y,a)$ (i.e., $y$ and $z$ are `unaccounted' child and grandchild of $a$, respectively). Then we proceed as follows:
\begin{itemize}
    \item Choose the least natural number $e_i\not\in \mathcal{R}_s$ and add $e_i$ to $\mathcal{R}$. Declare that $d_i$ is a child of $e_i$ and $e_i$ is a child of $\psi^{-1}(a)$.
    \item Put $\psi_{s+1}(d_i) = z$ and $\psi_{s+1}(e_i) = y$.
    \item If needed, extend the current $\mathcal{R}_{s+1}$ to `mimic' the finite tree $\mathcal{T}_{s+1}$.
    More formally, we add (the least unused) numbers to $\mathcal{R}_{s+1}$ in such a way that the current map $\psi_{s+1}$ could be extended to an isomorphism $\xi$ between the trees $\mathcal{R}_{s+1} \setminus \{ d_k : k\geq i+1\}$ and $\mathcal{T}_{s+1}$. Put $\psi_{s+1} = \xi$. Notice that now we have $\mathcal{T}_{s+1} \subseteq \mathrm{range}(\psi_{s+1})$.
\end{itemize}
If $\mathcal{T}_{s+1}$ does not have such nodes $y,z$, then go to stage $s+2$.

This concludes the description of the construction. Observe that the construction of the structure $\mathcal{R}= \bigcup_{s\in\mathbb{N}} \mathcal{R}_s$ is punctual. Let $\mathcal{A}$ be the connected component of $\mathcal{R}$ which contains $\psi^{-1}(a)$. The properties of the construction immediately imply that every node from $\mathcal{R} \setminus \{ d_i : i\in\mathbb{N}\}$ belongs to $\mathcal{A}$. In addition, the map $\psi = \bigcup_{s\in\mathbb{N}} \psi_s$ is an isomorphic embedding from $\mathcal{A}$ to $\mathcal{T}$.

Since the node $a$ has Property~($\ddagger$), there exist infinitely many stages $s+1$ such that the tree $\mathcal{T}_{s+1}$ contains the needed nodes $y$ and $z$. Therefore, we conclude that every element $d_i$ belongs to $\mathcal{A}$, and hence, we have $\mathcal{A} = \mathcal{R}$. In addition, $\mathcal{T} =\bigcup_{s\in\mathbb{N}} \mathcal{T}_s \subseteq \mathrm{range}(\psi)$. We deduce that $\psi$ is a computable isomorphism from $\mathcal{R}$ onto $\mathcal{T}$. Therefore, $\mathcal{R}$ is a punctual copy of the tree $\mathcal{T}$.

\proofsubparagraph{Case B.} Otherwise, we can choose a node $a\in \mathcal{T}$ such that $a$ has infinitely many children and only finitely many children of $a$ have their own children. We fix the list of all children of $a$ which have their own children: $u_0 <_{\mathcal{T}} u_1 <_{\mathcal{T}} \dots <_{\mathcal{T}} u_k$. Again, we assume that $a,u_0,u_1,\ldots,u_k\in\mathcal{T}_0$.

Without loss of generality, we may assume that the interval $(u_0; u_1)_{\mathcal{T}} = \{ x : u_0 <_{\mathcal{T}} x <_{\mathcal{T}} u_1\}$ is infinite. Notice that every element $x\in (u_0; u_1)_{\mathcal{T}}$ does not have children. By Theorem~1.5 of~\cite{kalimullin_algebraic_2017}, we can choose a punctual isomorphic copy $\mathcal{L} = (\mathbb{N}, <_{\mathcal{L}})$ of the linear ordering $( (u_0; u_1)_{\mathcal{T}}, <_{\mathcal{T}})$. 

The intuition behind the construction is as follows: in order to ensure the punctuality of the constructed structure $\mathcal{R}\cong \mathcal{T}$, we `quickly' produce a copy of the interval $(u_0;u_1)_{\mathcal{T}}$. This procedure works by promptly introducing fresh elements $d_0,d_1,d_2,\ldots$ which `copy' the punctual linear ordering $\mathcal{L}$.

Our construction proceeds in stages. At a stage $s$, we build a finite tree $\mathcal{R}_s$ and an isomorphism $\theta_s$ from $\mathcal{T}_s \setminus (u_0; u_1)_{\mathcal{T}}$ onto $\mathcal{R}_s\setminus (\theta(u_0); \theta(u_1))_{\mathcal{R}}$.

At stage $0$ we put $\mathcal{R}_0 = \mathcal{T}_0$ and $\theta_0 = \mathrm{id}\upharpoonright (\mathcal{T}_0\setminus (u_0; u_1)_{\mathcal{T}})$. 
At the end of each stage $s$, we choose the least natural number $x$ which currently does not belong to $\mathcal{R}_s$. We put $d_s = x$ and add $d_s$ into $\mathcal{R}_s$. We declare that:
\begin{itemize}
    \item $d_s$ is a child of $\theta(a)$,

    \item $d_s$ does not have any children in $\mathcal{R}$,

    \item $\theta(u_0) <_{\mathcal{R}} d_s <_{\mathcal{R}} \theta(u_1)$,

    \item $d_i <_{\mathcal{R}} d_j$ if and only if $i <_{\mathcal{L}} j$, for all $i,j\in\mathbb{N}$.
\end{itemize}
Recall that $\mathrm{card}(\mathcal{T}_{s+1} \setminus \mathcal{T}_s) \leq 1$.

Consider a stage $s+1$. Suppose that there exists a (unique) element $v \in \mathcal{T}_{s+1} \setminus \mathcal{T}_s$ additionally satisfying $v\not \in (u_0;u_1)_{\mathcal{T}}$. Then we choose the least unused number $y$ and add this $y$ to $\mathcal{R}_{s+1}$ in such a way that the map $\theta_{s+1} = \theta_s \cup \{ (v,y)\}$ becomes an isomorphism from $\mathcal{T}_{s+1} \setminus (u_0;u_1)_{\mathcal{T}}$ onto $\mathcal{R}_{s+1} \setminus (\theta(u_0);\theta(u_1))_{\mathcal{R}}$. If there is no such $v$, then proceed to stage $s+2$.

This concludes the description of the construction. Observe that the construction of $\mathcal{R} = \bigcup_{s\in\mathbb{N}} \mathcal{R}_s$ is punctual. In addition, the properties of the construction immediately imply that the map $\theta = \bigcup_{s\in\mathbb{N}} \theta_s$ is a computable isomorphism from the tree $\mathcal{T} \setminus (u_0;u_1)_{\mathcal{T}}$ onto $\mathcal{R}\setminus (\theta(u_0);\theta(u_1))_{\mathcal{R}}$.

Since the ordering $(\{ d_i : i\in\mathbb{N}\}, <_{\mathcal{R}})$ is isomorphic to $\mathcal{L} \cong ((u_0;u_1)_{\mathcal{T}}, <_{\mathcal{T}})$, we deduce that the structure $\mathcal{R}$ is a punctual copy of the tree $\mathcal{T}$.

    Now we show how the proofs of Case~A and Case~B could be adapted to provide P-TIME robustness. 
    Here we describe only the key modifications of the constructions for Cases~A and~B. 
    
    We need to build a P-TIME structure $\mathcal{R}$ such that $\mathcal{R}\cong \mathcal{T}$ and $\mathcal{R}$ has domain $\mathbb{B} = \{0,1\}^{<\omega}$. 
    As in the P-TIME construction of Theorem~\ref{ptime}, at each stage $s$ we perform a fixed number $l$ of steps of the algorithm which calculates the computable tree $\mathcal{T}$.

    (A)\ For $s\in\mathbb{N}$, the intended grandchild $d_s$ should be chosen as the $\sqsubset$-least string $\alpha$ which currently does not belong to $\mathcal{R}_s$. The element $d_s$ is then declared a `short' string, and short strings are essentially used to construct the `reservoir set' (in terminology of the construction of Theorem~\ref{ptime}). All the other nodes which are added to $\mathcal{R}_{s+1}$ (i.e., the node $e_i$ and nodes which help to mimic $\mathcal{T}_{s+1}$) must be chosen as long enough strings.

    (B)\ We note that Grigorieff~\cite{grigorieff_every_1990} proved that every infinite computable linear ordering has an isomorphic P-TIME copy with domain $\mathbb{B}$ (see also Theorem~4.8 in~\cite{Cenzer-Remmel-Survey-PTIME}). Hence, we can choose a P-TIME linear ordering $\mathcal{L} = (\mathbb{B}, <_{\mathcal{L}})$ which is isomorphic to the structure $((u_0;u_1)_{\mathcal{T}}, <_{\mathcal{T}})$. Observe that now the set $\mathbb{B}$ has two P-TIME orderings: $<_{\mathcal{L}}$ and $\sqsubset$.
    
    At a stage $s$, the node $d_s$ must be chosen as a short string (i.e., the $\sqsubset$-least string $\alpha$ which currently does not belong to $\mathcal{R}_s$). This $d_s$ will serve as the `copy' of the string $\beta^s$ which is the $s$-th least string \emph{with respect to $\sqsubset$} from $\mathrm{dom}(\mathcal{L}) = \mathbb{B}$. In order to copy the relation $<_{\mathcal{L}}$, we put $d_i <_{\mathcal{R}} d_j$ if and only if $\beta^i <_{\mathcal{L}} \beta^j$. The nodes $d_s$, $s\in\mathbb{N}$, are used to construct the reservoir set. All the other nodes from $\mathcal{R}$ (i.e., the nodes $y$ from the construction which are added to some $\mathcal{R}_{s+1}$) are chosen as long strings.

    In this case we are going to introduce a certain modification to the proof of the lemma that relations on this tree are P-TIME. We observe that the ordering on the interval $\mathcal{I}=(\theta(u_0);\theta(u_1))_{\mathcal{R}}$ (consisting of short strings) has some polynomial bound $V(n)$ and all the relations  on the part of the tree $\mathcal{R}$ outside of that interval (consisting of long strings) have some polynomial bound $W(n)$. The only case when a short and a long string are in a relation is when a short string from $\mathcal{I}$ is a child of a long string $\theta(a)$. All such cases are verifiable in time bound $V(n)$. 
    
    To decide if $a$ and $b$ are in some relation, with $lh(a) \leq lh(b)=n$ we perform the construction until $s_0$ -- the later of steps $V(n)$ and $W(n)$. The elements are in the relation only if this is established at the latest at step $s_0$. The proof that this can be done in polynomial time is very similar to the case of trees of unbounded depth.

    The described modifications could be formalized to provide a P-TIME structure $\mathcal{R}\cong \mathcal{T}$ with $\mathrm{dom}(\mathcal{R}) = \mathbb{B}$. 
\end{proof}
    

\end{document}